\pdfoutput=1
\documentclass[10pt,reqno,a4paper]{amsart}

\usepackage{iftex}
\ifPDFTeX
  \usepackage[utf8]{inputenc}
  \usepackage[T1]{fontenc}
\fi
\usepackage[english]{babel}
\usepackage[a4paper,left=0.98in,right=0.98in,top=1.02in,bottom=1.02in]{geometry}

\usepackage{microtype}
\usepackage{mathpazo}
\usepackage{euler}

\usepackage{amsmath,amssymb,amsthm,mathtools,mathrsfs,bm,eucal,tensor}
\usepackage{stmaryrd}
\usepackage{dsfont}
\usepackage{aliascnt}

\usepackage{graphicx}
\usepackage[dvipsnames,x11names]{xcolor}

\usepackage[all,cmtip]{xy}
\usepackage{tikz}
\usetikzlibrary{
  arrows.meta,
  calc,
  positioning,
  fit,
  backgrounds,
  shapes.geometric,
  decorations.markings
}
\usepackage{tikz-cd}

\tikzset{
  >=Stealth,
  box/.style={draw, rounded corners, inner sep=6pt},
  v/.style={draw, circle, inner sep=1.4pt},
  arr/.style={->, thick},
  darr/.style={->, thick, dashed},
  midarrow/.style={
    postaction={decorate},
    decoration={markings,mark=at position 0.58 with {\arrow{Stealth}}}
  }
}

\usepackage{enumitem}
\usepackage{comment,braket,xspace,csquotes}

\numberwithin{equation}{section}

\theoremstyle{plain}
\newtheorem{theorem}{Theorem}[section]

\newaliascnt{lemma}{theorem}
\newtheorem{lemma}[lemma]{Lemma}
\aliascntresetthe{lemma}

\newaliascnt{proposition}{theorem}
\newtheorem{proposition}[proposition]{Proposition}
\aliascntresetthe{proposition}

\newaliascnt{corollary}{theorem}
\newtheorem{corollary}[corollary]{Corollary}
\aliascntresetthe{corollary}

\theoremstyle{definition}

\newaliascnt{definition}{theorem}
\newtheorem{definition}[definition]{Definition}
\aliascntresetthe{definition}

\newaliascnt{notation}{theorem}
\newtheorem{notation}[notation]{Notation}
\aliascntresetthe{notation}

\newaliascnt{remark}{theorem}
\newtheorem{remark}[remark]{Remark}
\aliascntresetthe{remark}

\newaliascnt{assumption}{theorem}
\newtheorem{assumption}[assumption]{Assumption}
\aliascntresetthe{assumption}

\usepackage{hyperref}
\hypersetup{
  colorlinks=true,
  linkcolor=blue!60!black,
  citecolor=blue!60!black,
  urlcolor=blue!60!black
}
\usepackage[capitalize,nameinlink]{cleveref}

\Crefname{theorem}{Theorem}{Theorems}

\Crefname{lemma}{Lemma}{Lemmas}

\Crefname{proposition}{Proposition}{Propositions}

\Crefname{corollary}{Corollary}{Corollaries}

\Crefname{definition}{Definition}{Definitions}

\Crefname{notation}{Notation}{Notations}

\Crefname{remark}{Remark}{Remarks}

\Crefname{assumption}{Assumption}{Assumptions}

\DeclareMathOperator{\Spec}{Spec}
\DeclareMathOperator{\Proj}{Proj}
\DeclareMathOperator{\End}{End}
\DeclareMathOperator{\Sym}{Sym}
\DeclareMathOperator{\CH}{CH}
\DeclareMathOperator{\Ext}{Ext}

\newcommand{\Ad}{\mathrm{Ad}}

\newcommand{\C}{\mathbb{C}}
\newcommand{\R}{\mathbb{R}}
\newcommand{\Z}{\mathbb{Z}}
\newcommand{\NN}{\mathbb{N}}
\newcommand{\PP}{\mathbb{P}}
\newcommand{\RP}{\mathbb{RP}}
\newcommand{\Bl}{\mathrm{Bl}}
\newcommand{\OO}{\mathcal{O}}

\title[Geometry of the Donaldson--Friedman Pushout]{Geometry of the Donaldson--Friedman Pushout: Twistor degenerations and instanton charge}
\author{Amedeo Altavilla}
\address{Dipartimento di Matematica, Universit\`a degli Studi di Bari Aldo Moro, Via Edoardo Orabona, 4, 70125 Bari, Italy}
\email{amedeo.altavilla@uniba.it}

\author{Maur\'icio Corr\^ea}
\address{Dipartimento di Matematica, Universit\`a degli Studi di Bari Aldo Moro, Via Edoardo Orabona, 4, 70125 Bari, Italy}
\email{mauricio.barros@uniba.it}

\date{}

\begin{document}

 \begin{abstract}
We study the Donaldson--Friedman semistable twistor degeneration by combining the scheme-theoretic geometry of its Ferrand pushout with the logarithmic topology of its Kato--Nakayama realisation. For the central fibre
$
Z_0=\widetilde Z_1\cup_Q\widetilde Z_2,
$
the Ferrand description yields an explicit equaliser presentation of the operational Chow ring and a componentwise specialisation formula, with sharp restrictions on surfaces that glue across the exceptional quadric. The logarithmic structure of the same normal-crossing fibre retains data not visible in ordinary intersection theory: after fixing the phase of the smoothing parameter, the Kato--Nakayama space over $Q$ is the unit circle bundle of the normal line bundle, and over a ruling fibre its anti-diagonal quotient is diffeomorphic to $\mathbb{RP}^3$. Restriction to curves in $Q$ gives a log-topological refinement of the algebraic intersection data. For bundles $E_0$ on $Z_0$ obtained by gluing Ward or Hartshorne--Serre data from the two components, we prove additivity of the second Chern cycle $c_2(E_0)\cap[Z_0]$. If $E_0$ extends over the semistable smoothing, its polarised charge is the sum of the component charges; under the usual reality and triviality conditions of the Ward correspondence, the bundle on a smooth fibre determines an anti-self-dual $SU(2)$-instanton on the connected sum of the underlying four-manifolds.
\end{abstract}
\maketitle

\section{Introduction}

Twistor theory relates four-dimensional differential geometry to
complex geometry by encoding geometric and gauge-theoretic structures
on an oriented four-manifold $X$ in a complex threefold $Z$, its
twistor space. In particular, the Atiyah--Ward correspondence
\cite{AtiyahWard1977} describes anti-self-dual Yang--Mills instantons
in terms of holomorphic bundles on $Z$, with the instanton number
recorded by the second Chern class.

The Donaldson--Friedman construction starts with twistor spaces
$Z_1,Z_2$ and chosen twistor lines $\ell_i\subset Z_i$. Blowing up
along $\ell_i$ gives $\widetilde Z_i=\Bl_{\ell_i}(Z_i)$ with
exceptional quadrics $Q_i\simeq\PP^1\times\PP^1$. After identifying
$Q_1$ and $Q_2$ by an isomorphism $\sigma$ exchanging the two rulings, one obtains
\[
Z_0=\widetilde Z_1\cup_Q\widetilde Z_2,
\qquad Q\simeq Q_1\simeq Q_2,
\]
the central fibre of a semistable smoothing $\pi:\mathcal Z\to\Delta$
with local model $t=uv$ along $Q$. For $t\neq0$, the smooth fibre
$Z_t$ is a twistor space of $X_1\#X_2$ \cite{DF,Clemens,Clemens1983}.
Scheme-theoretically, $Z_0$ is a Ferrand pushout \cite{Fer03}. We study
this singular fibre directly, rather than merely as an intermediate
stage in the smoothing.

The normal-crossing fibre carries two complementary structures dictated
by the same semistable equation $t=uv$. Scheme-theoretically, the Ferrand
pushout expresses $Z_0$ as the gluing of two smooth components along $Q$.
This makes its operational intersection theory accessible through descent
to the components and turns the matching problem into a calculation on the
double locus. Logarithmically, the semistable structure retains the
boundary data of the two branches, and the Kato--Nakayama realisation
records the angular information that disappears on passing to the
underlying central fibre \cite{KatoNakayama1999}. More generally,
Kato--Nakayama spaces recover the topology of log-smooth degenerations
from their logarithmic special fibres \cite{AchingerOgus2020}, and have
been used to describe explicitly the topology of reducible toric
degenerations \cite{Arguz2021}.

Our approach exploits both structures on the same Donaldson--Friedman
fibre. The Ferrand description, together with operational Chow descent,
controls limiting algebraic cycles and their compatibility along $Q$.
The logarithmic structure supplies the topological counterpart of this
matching problem: after fixing the phase of the smoothing parameter, its
Kato--Nakayama realisation retains the phase with which the two branches
approach the double locus. In this setting both sides are explicit. The
normal geometry of $Q$ governs the restrictions of Chow classes and also
determines the fixed-phase circle bundle. Restricting this bundle to a
limiting curve $c\subset Q$ gives a principal circle bundle over $c$, and
therefore refines the ordinary algebraic trace by log-topological data.
Thus the algebraic and logarithmic constructions are not independent
additions to the degeneration: they give complementary invariants of the
same singular twistor fibre and meet again along its double locus.

\begin{samepage}
The preceding point of view leads to three groups of results.

\medskip

\noindent\textbf{Algebraic structure: operational Chow theory and specialisation.}
Since $Z_0$ is singular, its natural cohomological intersection theory is
Fulton's operational Chow ring. We prove the explicit equaliser description
\[
A^\bullet(Z_0)\cong
\Bigl\{(\alpha_1,\alpha_2)\in
\CH^\bullet(\widetilde Z_1)\oplus\CH^\bullet(\widetilde Z_2)
\ \Big|\ j_1^*\alpha_1=\sigma^*j_2^*\alpha_2\Bigr\};
\]
see \Cref{thm:CH-pushout-generators-relations}.
\end{samepage}
Thus a global operational
class is determined by ordinary Chow classes on the two smooth branches,
subject only to agreement on the double locus. Fulton's specialisation map
is compatible with this description componentwise: a cycle on a smooth
fibre specialises to the sum of its traces on the two branches, and their
refined traces agree on $Q$; see
\Cref{prop:sp-decomp,cor:sp-by-components}.

This description has a rigid geometric consequence. For irreducible reduced
surfaces satisfying the hypotheses of \Cref{Thm2}, compatibility of the
trace classes on $Q$ leaves only two possibilities:
\[
(d_1,d_2)=(2,2)\qquad\text{or}\qquad(d_1,d_2)=(1,1).
\]
In particular, the gluing condition cannot hold if
$\max\{d_1,d_2\}\geq3$; see \Cref{prop:surface-trace,Thm2}. Hence the
equaliser condition is not merely a description of the ring: it gives
explicit restrictions on algebraic subvarieties that can occur across the
Donaldson--Friedman central fibre.

\medskip

\noindent\textbf{Logarithmic topology of the neck.}
The logarithmic structure is used here not merely as a language for the
normal-crossing fibre. It provides the topological counterpart of the
algebraic matching problem. Operational Chow theory detects the trace of a
limiting cycle on $Q$, whereas the Kato--Nakayama realisation retains the
phase data of the two branches along that trace. The same semistable
equation $t=uv$ that governs algebraic matching also governs this angular
topology. Fixing the phase $\theta\in S^1$ of the smoothing parameter, we
identify the corresponding branchwise part of the Kato--Nakayama space over
$Q$ with the unit circle bundle
\[
Q^{\log}\big|_\theta\simeq
S\!\left(\mathcal N_{Q/\widetilde Z_1}\right)\longrightarrow Q.
\]
The fixed-phase neck retracts onto this bundle; see
\Cref{prop:strong-def-retract-fixed-theta-rewrite}. On a ruling fibre
$F\simeq\PP^1$, its restriction is the Hopf circle bundle and therefore has
total space $S^3$. Combining it with the Hopf fibration over $F$ and taking
the anti-diagonal circle quotient gives a principal circle bundle of first
Chern class $2$, whose total space is
\[
L(2,1)\simeq\RP^3;
\]
see \Cref{thm:DF-diagram-KN-Clemens}. This $\RP^3$ is auxiliary and is not
the $S^3$ along which the connected sum of the underlying four-manifolds is
performed. Restricting the fixed-phase circle bundle to curves in $Q$ gives
a logarithmic refinement of their ordinary intersection classes; see
\Cref{prop:KN-circle-on-c}.

\medskip

\noindent\textbf{Bundle-theoretic applications and charge.}
We finally apply these descriptions to bundles arising from Ward and
Hartshorne--Serre data. In both constructions the second Chern cycle is
additive. In the Hartshorne--Serre notation, for the glued rank-$2$ bundle
$E_0$ and its restrictions $E_i$ one has
\[
c_2(E_0)\cap[Z_0]
=
(\phi_1)_*\bigl(c_2(E_1)\cap[\widetilde Z_1]\bigr)
+
(\phi_2)_*\bigl(c_2(E_2)\cap[\widetilde Z_2]\bigr).
\]
The Ward construction satisfies the same formula; see
\Cref{prop:no-neck-term-ward,prop:c2-pushout-general}. If the glued bundle
extends to a bundle on the semistable family and $H_0$ is the limiting
polarisation, with $H_i=\phi_i^*H_0$, then for $t\neq0$ sufficiently small
its polarised charge satisfies
\[
\mathrm{charge}_H(E_t)
=
\deg\bigl((c_2(E_1)H_1)\cap[\widetilde Z_1]\bigr)
+
\deg\bigl((c_2(E_2)H_2)\cap[\widetilde Z_2]\bigr);
\]
see \Cref{prop:no-neck-term-ward,prop:HS-charge-additivity}. Thus no
additional term supported on the double locus occurs in the charge formula.
Under the usual reality, triviality-on-real-twistor-lines, and determinant
conditions of the Ward correspondence, the smoothed rank-$2$ bundle
determines an anti-self-dual $SU(2)$-connection on $X_1\#X_2$; see
\Cref{cor:instantons-HS}.

These results complement Honda's deformation-theoretic treatment of
Donaldson--Friedman triples \cite{Honda1999,Honda2003}, where the singular
fibre is studied primarily through its deformation to a smooth twistor
space. Our emphasis is different: the same central fibre is treated both
as a singular scheme and as a logarithmic special fibre. Kimura's descent
theorem \cite{Kimura1992} supplies the algebraic mechanism behind the
operational Chow calculation, whereas Kato--Nakayama theory supplies the
log-topological realisation of the semistable structure
\cite{KatoNakayama1999,AchingerOgus2020,Arguz2021}. We are not aware of a
previous treatment of the Donaldson--Friedman degeneration in which these
two descriptions are used together to obtain, from the same central fibre,
both explicit algebraic matching data and the topology of the smoothing
neck.

Section~\ref{sec:CH-pushout} develops the operational Chow theory and
specialisation, Section~\ref{sec:KN} the logarithmic topology of the
neck, and Sections~\ref{sec:instantons-pushout} and
\ref{sec:serre-pushout-instantons} the bundle and charge applications.

\section{The operational Chow ring of the pushout}\label{sec:CH-pushout}

Let $X_1,X_2$ be compact anti-self-dual $4$--manifolds, let
$\pi_i:Z_i\to X_i$ be their twistor fibrations, and fix points
$x_i\in X_i$. We denote by
$
\ell_i:=\pi_i^{-1}(x_i)\subset Z_i
$
the corresponding twistor lines. Let
$
f_i:\widetilde Z_i:=\Bl_{\ell_i}(Z_i)\longrightarrow Z_i
$
be the blow-up along $\ell_i$, and let
$j_i:Q_i\hookrightarrow \widetilde Z_i$
be the exceptional divisor.
In the twistor setting one has the classical splitting
$
\mathcal N_{\ell_i/Z_i}\simeq \OO_{\PP^1}(1)\oplus \OO_{\PP^1}(1),
$
so the exceptional divisor is
$
Q_i=\PP(\mathcal N_{\ell_i/Z_i})\simeq \PP^1\times \PP^1.
$
The two rulings on $Q_i$ arise from the projection
$g_i:Q_i\to \ell_i\simeq \PP^1$ and from the lines in the normal
fibres.
Fix an isomorphism
\begin{equation}\label{eq:sigma-def}
\sigma:Q_1\xrightarrow{\sim}Q_2
\end{equation}
exchanging the two rulings and commuting with the real structures.

\begin{definition}
The \emph{Ferrand pushout} (\cite{Fer03}) associated with the closed immersions
\begin{equation*}\label{eq:ferrand-immersions}
j_1:Q_1\hookrightarrow \widetilde Z_1,
\qquad
j_2:Q_2\hookrightarrow \widetilde Z_2,
\end{equation*}
and with the identification $\sigma:Q_1\xrightarrow{\sim}Q_2$, is a
scheme $\widetilde Z$ obtained by gluing $\widetilde Z_1$ and
$\widetilde Z_2$ along the exceptional divisors $Q_1$ and $Q_2$ via
$\sigma$. More precisely, $\widetilde Z$ comes equipped with closed
immersions
\begin{equation*}\label{eq:phi-immersions}
\phi_1:\widetilde Z_1\hookrightarrow \widetilde Z,
\qquad
\phi_2:\widetilde Z_2\hookrightarrow \widetilde Z,
\end{equation*}
such that
$
\phi_1\circ j_1=\phi_2\circ j_2\circ \sigma,
$
and is universal for this gluing.
\end{definition}

In the Donaldson--Friedman construction, the singular central fibre is
precisely this Ferrand pushout:
\begin{equation*}\label{eq:Z-pushout}
\widetilde Z:=\widetilde Z_1\cup_Q \widetilde Z_2,
\qquad
Q\simeq Q_1\simeq Q_2.
\end{equation*}
The Ferrand pushout $\widetilde Z$ is locally analytically isomorphic
to the standard normal crossing model
$
\{uv=0\}\subset \C^4,$
with components $\{u=0\}$ and $\{v=0\}$ and double locus
$\{u=v=0\}$. In particular, $\widetilde Z$ is a simple normal
crossings threefold.

Because $\widetilde Z$ is singular, the natural intersection-theoretic
ring on $\widetilde Z$ is the \emph{operational Chow ring}
$A^\bullet(\widetilde Z)$ in the sense of \cite[Chapter~17]{Fulton}.
For a singular scheme $X$, the groups $\CH_*(X)$ need not carry the
contravariant ring structure available in the smooth case. Accordingly, the relevant cohomological theory is Fulton's operational Chow ring $A^\bullet(X)$
\cite[Chapter~17]{Fulton}; when $X$ is smooth, cap product with the
fundamental class identifies $A^\bullet(X)$ with $\CH^\bullet(X)$.
For the smooth varieties $\widetilde Z_i$ and $Q_i$, the canonical identifications are
\[
A^\bullet(\widetilde Z_i)=\CH^\bullet(\widetilde Z_i),\qquad
A^\bullet(Q_i)=\CH^\bullet(Q_i).
\]

For each $i=1,2$, let
\begin{equation}\label{eq:bi-def}
b_i:=g_i^*c_1\bigl(\OO_{\ell_i}(1)\bigr)\in\CH^1(Q_i),
\end{equation}
and let $w_i\in\CH^1(Q_i)$ be the class of a fibre of the other
ruling. Set
\begin{equation*}\label{eq:zi-def}
z_i:=b_i-w_i,
\qquad\text{so that}\qquad
b_i=z_i+w_i.
\end{equation*}
Since under $Q_i\simeq\PP^1\times\PP^1$ the tautological bundle
$\OO_{Q_i}(1)$ identifies with $\OO_{Q_i}(1,1)$, its first Chern
class is
\begin{equation}\label{eq:xi-zw}
\xi_i:=c_1\bigl(\OO_{Q_i}(1)\bigr)=b_i+w_i=z_i+2w_i.
\end{equation}

\begin{lemma}\label{lem:CH-abelian-pushout}
Let
$
\nu:\widetilde Z^\nu:=\widetilde Z_1\sqcup \widetilde Z_2
\longrightarrow \widetilde Z
$
be the normalisation map, and let $i:Q\hookrightarrow \widetilde Z$
be the inclusion of the double locus. Then $\nu$ is an envelope, it is
an isomorphism over $\widetilde Z\setminus Q$, and
$E:=\nu^{-1}(Q)=Q_1\sqcup Q_2$.
Consequently, Kimura's exact sequence for operational Chow cohomology
gives an exact sequence
\begin{equation}\label{eq:operational-descent-seq}
0\longrightarrow A^\bullet(\widetilde Z)
\xrightarrow{(\nu^*,\,i^*)}
\CH^\bullet(\widetilde Z_1)\oplus
\CH^\bullet(\widetilde Z_2)\oplus
\CH^\bullet(Q_1)
\xrightarrow{\ d\ }
\CH^\bullet(Q_1)\oplus \CH^\bullet(Q_2),
\end{equation}
where
\begin{equation}\label{eq:d-def}
d(\alpha_1,\alpha_2,\beta)
=
\bigl(j_1^*\alpha_1-\beta,\ \sigma^*j_2^*\alpha_2-\beta\bigr).
\end{equation}

\end{lemma}

\begin{proof}
The map $\nu$ is finite, hence proper, and is an isomorphism over $\widetilde Z\setminus Q$. We claim that it is an
envelope. Let $V\subset \widetilde Z$ be an irreducible closed
subvariety.
If $V\subset Q$, let $V_1\subset Q_1\subset \widetilde Z^\nu$ be the
corresponding closed subvariety under the identification
$Q_1\simeq Q$. Then the induced morphism $V_1\to V$ is an isomorphism,
hence birational. If $V\not\subset Q$, then the generic point of $V$
lies on exactly one irreducible component of $\widetilde Z$, say
$\widetilde Z_1$ or $\widetilde Z_2$, and the closure of that generic
point in the corresponding branch gives a closed subvariety of
$\widetilde Z^\nu$ mapping birationally onto $V$. Thus $\nu$ is an
envelope.
Applying Kimura's exact sequence \cite{Kimura1992} for operational
Chow cohomology to the envelope $\nu$, with centre $Q\subset \widetilde Z$ and inverse
image $E=Q_1\sqcup Q_2$, yields
\[
0\longrightarrow A^\bullet(\widetilde Z)
\longrightarrow
A^\bullet(\widetilde Z^\nu)\oplus A^\bullet(Q)
\longrightarrow
A^\bullet(E).
\]
More explicitly, the second arrow sends an operational class
$\gamma\in A^\bullet(\widetilde Z)$ to
\[
\bigl(\nu^*\gamma,\ i^*\gamma\bigr),
\]
while the third arrow is the difference of the two pull-backs from
$\widetilde Z^\nu$ and from $Q$ to $E=Q_1\sqcup Q_2$. Under these identifications, this is the map $d$ in
\eqref{eq:d-def}.
Using
\[
A^\bullet(\widetilde Z^\nu)
=
\CH^\bullet(\widetilde Z_1)\oplus \CH^\bullet(\widetilde Z_2),
\qquad
A^\bullet(E)
=
\CH^\bullet(Q_1)\oplus \CH^\bullet(Q_2),
\]
and identifying $A^\bullet(Q)$ with $\CH^\bullet(Q_1)$ via the chosen
copy $Q=Q_1$, one gets \eqref{eq:operational-descent-seq} with the map
\eqref{eq:d-def}. This proves the asserted descent sequence.
\end{proof}

\begin{lemma}
\label{lem:chow-quadric}
For each $k=1,2$, the Chow ring of $Q_k\simeq\PP^1\times\PP^1$ is
\begin{equation}\label{eq:quadric-ring-bw}
\CH^\bullet(Q_k)\cong
\frac{\Z[b_k,w_k]}{\langle\, b_k^2,\ w_k^2\,\rangle},
\end{equation}
where $b_k$ is the class of a fibre of $g_k:Q_k\to\ell_k$ and $w_k$
is the class of a fibre of the opposite ruling. After the change of
variables $z_k:=b_k-w_k$,
\begin{equation}\label{eq:quadric-ring-pres}
\CH^\bullet(Q_k)\cong
\frac{\Z[z_k,w_k]}{\langle\, w_k^2,\ z_k^2+2z_kw_k\,\rangle},
\end{equation}
equivalently
\begin{equation*}\label{eq:quadric-relations}
w_k^2=0,\qquad (z_k+w_k)^2=0,\qquad z_k^2=-2z_kw_k.
\end{equation*}
The tautological class satisfies
\begin{equation*}\label{eq:xi-taut}
\xi_k:=c_1\bigl(\OO_{Q_k}(1)\bigr)=b_k+w_k=z_k+2w_k.
\end{equation*}
The isomorphism $\sigma:Q_1\xrightarrow{\sim}Q_2$ exchanging the two rulings
acts by
\begin{equation}\label{eq:sigma-on-w}
\sigma_*(w_1)=b_2=z_2+w_2,\qquad\sigma_*(b_1)=w_2,
\end{equation}
equivalently
\begin{equation}\label{eq:sigma-on-z}
\sigma_*(z_1)=-z_2,\qquad\sigma_*(\xi_1)=\xi_2.
\end{equation}
\end{lemma}

\begin{proof}
Since $Q_k\simeq\PP^1\times\PP^1$, one has $b_k^2=0$, $w_k^2=0$,
$b_kw_k=[\mathrm{pt}]$, giving \eqref{eq:quadric-ring-bw}.
Substituting $b_k=z_k+w_k$ into $b_k^2=0$ gives
$z_k^2+2z_kw_k+w_k^2=0$, hence $z_k^2+2z_kw_k=0$ since $w_k^2=0$,
which is \eqref{eq:quadric-ring-pres}.
Under $Q_k=\PP(\OO_{\PP^1}(1)\oplus\OO_{\PP^1}(1))$, the tautological
bundle is $\OO_{Q_k}(1,1)$, giving $\xi_k=b_k+w_k=z_k+2w_k$.
Since $\sigma$ exchanges the two rulings, $\sigma_*(w_1)=b_2$ and
$\sigma_*(b_1)=w_2$. Writing $b_2=z_2+w_2$ gives
\eqref{eq:sigma-on-w}. Subtracting gives $\sigma_*(z_1)=-z_2$, and
adding gives $\sigma_*(\xi_1)=\xi_2$.
\end{proof}

\begin{proposition}\label{prop:blowup-decomp}
For each $k=1,2$, the Chow groups of the blown-up space
$\widetilde Z_k=\Bl_{\ell_k}(Z_k)$ decompose as
\begin{equation}\label{eq:blowup-chow-decomp}
\CH^p(\widetilde Z_k)\cong
f_k^*\CH^p(Z_k)\oplus (j_k)_*g_k^*\CH^{p-1}(\ell_k).
\end{equation}
Since $\ell_k\simeq\PP^1$, the exceptional summand contributes:
\begin{itemize}
\item in degree $p=1$: one generator $[Q_k]=c_1(\OO_{\widetilde Z_k}(Q_k))$;
\item in degree $p=2$: one generator $(j_k)_*(b_k)$;
\item in degree $p\ge 3$: nothing.
\end{itemize}
Moreover, in $\CH^2(\widetilde Z_k)$:
\begin{equation}\label{eq:exceptional-line-class}
(j_k)_*(\xi_k)=f_k^*(i_k)_*[\ell_k],
\end{equation}
equivalently $(j_k)_*(z_k+2w_k)=f_k^*(i_k)_*[\ell_k]$,
where $i_k:\ell_k\hookrightarrow Z_k$ is the inclusion.
\end{proposition}

\begin{proof}
The decomposition \eqref{eq:blowup-chow-decomp} is the standard
blow-up formula for a smooth threefold along a smooth codimension-$2$
centre \cite[Theorem~6.7 and Example~15.4.2]{Fulton}. The generators
of the exceptional summand in each degree follow from
$\CH^{p-1}(\PP^1)\cong\Z$ for $p-1\in\{0,1\}$ and $0$ otherwise.
The identity \eqref{eq:exceptional-line-class} is
\cite[Example~15.4.2]{Fulton}: for a blow-up along a codimension-$2$
centre, the pull-back of the class of the centre equals the push-forward
of the tautological class on the exceptional divisor, i.e.\
$(j_k)_*\xi_k=f_k^*(i_k)_*[\ell_k]$.
\end{proof}

\begin{theorem}
\label{thm:CH-pushout-generators-relations}
With the notation of \Cref{lem:chow-quadric} and
\Cref{prop:blowup-decomp}, the following hold.

\begin{enumerate}

\item\label{item:gluing-relations}
The operational Chow ring of the singular pushout is given by
\begin{equation}\label{eq:A-equalizer-theorem}
A^\bullet(\widetilde Z)\cong
\Bigl\{
(\alpha_1,\alpha_2)\in
\CH^\bullet(\widetilde Z_1)\oplus \CH^\bullet(\widetilde Z_2)
\ \Big|\
j_1^*\alpha_1=\sigma^*j_2^*\alpha_2
\text{ in }\CH^\bullet(Q_1)
\Bigr\}.
\end{equation}
In particular, the pair $([Q_1],[Q_2])$ defines a class
$[Q]\in A^1(\widetilde Z)$, because
\begin{equation}\label{eq:Q-self-restriction}
j_1^*[Q_1]=-\xi_1=\sigma^*(-\xi_2)=\sigma^*j_2^*[Q_2].
\end{equation}
More generally, a pair of component classes defines an operational
class on $\widetilde Z$ if and only if its restrictions to the double
locus agree under $\sigma$.

\item\label{item:double-locus-ring}
The Chow ring of the double locus $Q$ is identified with either
branch:
\begin{equation*}\label{eq:double-locus-ring}
\CH^\bullet(Q)\cong\CH^\bullet(Q_1)\cong\CH^\bullet(Q_2)
\cong\frac{\Z[z,w]}{\langle\,w^2,\,z^2+2zw\,\rangle},
\end{equation*}
with the two branch descriptions related by
\eqref{eq:sigma-on-w}--\eqref{eq:sigma-on-z}.

\end{enumerate}
\end{theorem}

\begin{proof}
For Item~\ref{item:gluing-relations}, eliminate the intermediate class $\beta$ from the descent sequence of \Cref{lem:CH-abelian-pushout}. This gives \eqref{eq:A-equalizer-theorem}; the product is componentwise because the pull-back maps in the descent sequence are ring homomorphisms.
For the pair $([Q_1],[Q_2])$, since $Q_k$ is a Cartier divisor in
$\widetilde Z_k$, one has $j_k^*[Q_k]=c_1(\mathcal
N_{Q_k/\widetilde Z_k})=-\xi_k$ (the normal bundle computation is given in \Cref{lem:normal-exceptional}). Since $\sigma^*\xi_2=\xi_1$,
we get \eqref{eq:Q-self-restriction}, so the pair is compatible.

Item \ref{item:double-locus-ring}: after identifying $Q$ with $Q_1$,
its Chow ring is the ring in \eqref{eq:quadric-ring-pres} of
\Cref{lem:chow-quadric}. The identification with $\CH^\bullet(Q_2)$
is via transport through $\sigma$, with change of basis
\eqref{eq:sigma-on-w}--\eqref{eq:sigma-on-z}.
\end{proof}

\subsection{Gluing surfaces in the pushout}
%\label{subsec:gluing-surfaces}
%======================================================================

We next determine how surfaces meet the exceptional quadrics. Since the central fibre is singular, the relevant traces are computed on the smooth quadrics $Q_i\simeq\PP^1\times\PP^1$, where ordinary Chow theory applies.

\begin{definition}
Let $\pi:Z\to X$ be a twistor space, and let $S\subset Z$ be an
irreducible surface. Assume that $S$ is generically transverse to the
twistor fibration and that $S\in |E|$ for some line bundle $E$ on $Z$.
The \emph{twistor degree} of $S$ is the integer
\begin{equation*}\label{eq:twistor-degree-def}
\deg_{\mathrm{tw}}(S):=\deg\bigl(c_1(E)|_L\bigr),
\end{equation*}
where $L$ is any twistor line not contained in $S$.
\end{definition}

The definition is well posed because twistor lines are numerically equivalent.
In the basic algebraic examples it recovers the usual degree: for
$Z=\PP^3$ it is the projective degree, while on the flag threefold a
divisor in $|\OO(d_1,d_2)|$ has twistor degree $d_1+d_2$.

\begin{proposition}
\label{prop:surface-trace}
Let $f:\widetilde Z=\Bl_\ell(Z)\to Z$ be the blow-up of a twistor
space along a twistor line $\ell$, with exceptional divisor
$Q=f^{-1}(\ell)\simeq\PP^1\times\PP^1$. Use the classes $b,w,z,\xi$
on $Q$ from \eqref{eq:bi-def}--\eqref{eq:xi-zw}.
Let $S\subset Z$ be an irreducible surface of twistor degree
$d:=\deg_{\mathrm{tw}}(S)$, and let $\widetilde S$ be its strict
transform. The blow-up does not change the twistor degree:
$\deg_{\mathrm{tw}}(\widetilde S)=d$. The trace $[\widetilde S\cap Q]$
in $\CH^1(Q)$ is determined as follows.

\begin{enumerate}
\item\label{item:surface-a}
If $\ell\not\subset S$, then
$
[\widetilde S\cap Q]=d\,b=d(z+w).
$

\item\label{item:surface-b}
If $\ell\subset S$ and $S$ is smooth along $\ell$, then
$\widetilde S\cap Q$ is a section of $g:Q\to\ell$ and
\begin{equation*}\label{eq:trace-with-ell}
[\widetilde S\cap Q]=(d-1)b+w,
\end{equation*}
which has degree $1$ over $\ell$:
$\bigl((d-1)b+w\bigr)\cdot b=1$.

\item\label{item:surface-c}
An irreducible reduced curve $c\subset Q$ of class $(d-1)b+w$ is
rational and meets a generic twistor line in exactly $d$ points.
\end{enumerate}
\end{proposition}

\begin{proof}
For a generic twistor line $L\subset Z$ disjoint from $\ell$, the
map $f$ is an isomorphism near $L$, so
$\deg_{\mathrm{tw}}(\widetilde S)=\deg([S]\cdot[L])=d$.

\smallskip
\noindent\emph{Proof of \ref{item:surface-a}.}
Since $\ell\not\subset S$, the scheme-theoretic intersection
$S\cap\ell$ is zero-dimensional of length $d$, counted with
multiplicity: this is precisely the degree of the line bundle $E$
restricted to the twistor line $\ell$. Each point of $S\cap\ell$,
with its scheme-theoretic multiplicity, contributes one fibre of
$g:Q\to\ell$, of class $b$. Summing over all points with
multiplicities gives $[\widetilde S\cap Q]=d\,b$.

\smallskip
\noindent\emph{Proof of \ref{item:surface-b}.}
Since $\ell\subset S$ and $S$ is smooth along $\ell$, the strict
transform $\widetilde S$ meets $Q=\PP(\mathcal N_{\ell/Z})\to\ell$
along the section induced by the surjection
$\mathcal N_{\ell/Z}\twoheadrightarrow\mathcal N_{S/Z}|_\ell$. Since
$S\in|E|$ and $S$ is smooth along $\ell$, one has
$\mathcal N_{S/Z}|_\ell\simeq E|_\ell\simeq\OO_\ell(d)$. Twisting
by $\OO_\ell(-1)$ and using
$\mathcal N_{\ell/Z}\simeq\OO_\ell(1)^{\oplus 2}$, one gets a
quotient $\OO_\ell^{\oplus 2}\twoheadrightarrow\OO_\ell(d-1)$. Under
$Q\simeq\ell\times\PP^1$, the corresponding section is the graph of a
map $\PP^1\to\PP^1$ of degree $d-1$, with divisor class $(d-1)b+w$.
Since $bw=[\mathrm{pt}]$, one computes $((d-1)b+w)\cdot b=1$.

 \smallskip
\noindent\emph{Proof of \ref{item:surface-c}.}
Let $D\subset Q\simeq\PP^1\times\PP^1$ be a divisor of bidegree
$(m,n)$, so that
\[
[D]=m\,b+n\,w \in \CH^1(Q).
\]
Let \(p_a(D)\) denote the arithmetic genus of \(D\). Since
\[
K_Q=-2b-2w,
\qquad
b^2=w^2=0,
\qquad
bw=[\mathrm{pt}],
\]
the adjunction formula gives
\[
p_a(D)=\frac{D\cdot(D+K_Q)}{2}+1.
\]
Now,
$
D^2=(m\,b+n\,w)^2=2mn,$
and
 $
D\cdot K_Q=(m\,b+n\,w)\cdot(-2b-2w)=-2m-2n.
$
Therefore
$$
p_a(D)=\frac{2mn-2m-2n}{2}+1=(m-1)(n-1).
$$
Applying this to
$
D=(d-1)b+w,
$
that is, $(m,n)=(d-1,1)$, we obtain
$
p_a(D)=0.
$
Since $c$ is irreducible and reduced, it follows that $c$ is rational.
Its intersection with a twistor line of class $\xi=b+w$ is
\[
((d-1)b+w)\cdot(b+w)=d.
\]

\end{proof}

\begin{theorem} 
\label{Thm2}
In the Donaldson--Friedman construction, let $S_i\subset Z_i$ be irreducible
reduced surfaces of twistor degree $d_i\ge 1$, and let
$\widetilde S_i\subset \widetilde Z_i$ be their strict transforms.
Assume that whenever $\ell_i\subset S_i$, the surface $S_i$ is smooth
along $\ell_i$.
By \Cref{prop:surface-trace}, the trace classes
\[
[\widetilde S_i\cap Q_i]\in \CH^1(Q_i)
\]
are well defined. Assume moreover that they satisfy the gluing condition
\begin{equation*}\label{eq:gluing-trace-condition}
\sigma_*\bigl([\widetilde S_1\cap Q_1]\bigr)
=
[\widetilde S_2\cap Q_2]
\qquad\text{in }\CH^1(Q_2).
\end{equation*}
Then exactly one of the following alternatives occurs:
\begin{enumerate}
\item\label{item:case-both-in}
both $\ell_1\subset S_1$ and $\ell_2\subset S_2$, and necessarily
$d_1=d_2=2$;

\item\label{item:case-one-in}
exactly one of the two inclusions $\ell_1\subset S_1$ and
$\ell_2\subset S_2$ holds, and necessarily $d_1=d_2=1$.
\end{enumerate}
In particular, the gluing condition cannot hold if neither
$\ell_1\subset S_1$ nor $\ell_2\subset S_2$, and it cannot hold if
$\max\{d_1,d_2\}\ge 3$.
Moreover, if $\ell_1\subset S_1$, then
\[
\sigma_*\bigl([\widetilde S_1\cap Q_1]\bigr)=(d_1-1)w_2+b_2,
\]
and this class has degree $d_1-1$ with respect to the ruling
$Q_2\to \ell_2$. The analogous statement for $\ell_2\subset S_2$
is obtained by symmetry.
\end{theorem}

\begin{proof}
By \Cref{prop:surface-trace}, for each $i=1,2$ one has
\[
[\widetilde S_i\cap Q_i]
=
\begin{cases}
d_i\,b_i, & \text{if }\ell_i\not\subset S_i,\\[4pt]
(d_i-1)b_i+w_i, & \text{if }\ell_i\subset S_i.
\end{cases}
\]
We compare these classes using
\(\sigma_*(b_1)=w_2\), and
\(\sigma_*(w_1)=b_2\).
Since $Q_2\simeq \PP^1\times\PP^1$, one has
\[
\CH^1(Q_2)\simeq \Z\,b_2\oplus \Z\,w_2,
\]
so $b_2$ and $w_2$ are linearly independent.
Assume first that both $\ell_1\subset S_1$ and $\ell_2\subset S_2$.
Then
\[
\sigma_*\bigl((d_1-1)b_1+w_1\bigr)=(d_1-1)w_2+b_2.
\]
Imposing equality with $(d_2-1)b_2+w_2$ gives
\[
(d_1-1)w_2+b_2=(d_2-1)b_2+w_2.
\]
Since $b_2$ and $w_2$ are independent in $\CH^1(Q_2)$, we obtain
$
d_1-1=1, 
$ and $
d_2-1=1,
$
hence $d_1=d_2=2$. This proves
\ref{item:case-both-in}.
Assume next that $\ell_1\subset S_1$ and $\ell_2\not\subset S_2$.
Then
\[
\sigma_*\bigl((d_1-1)b_1+w_1\bigr)=(d_1-1)w_2+b_2,
\]
and the gluing condition becomes
$
(d_1-1)w_2+b_2=d_2b_2.
$
Again by independence of $b_2$ and $w_2$, this yields
\[
d_1-1=0,
\qquad
d_2=1,
\]
so $d_1=d_2=1$.
Similarly, if $\ell_1\not\subset S_1$ and $\ell_2\subset S_2$, then
$
\sigma_*(d_1b_1)=d_1w_2,
$
and the gluing condition becomes
$
d_1w_2=(d_2-1)b_2+w_2.
$
Therefore
\[
d_1=1,
\qquad
d_2-1=0,
\]
hence again $d_1=d_2=1$. This proves
\ref{item:case-one-in}.

Finally, if neither $\ell_1\subset S_1$ nor $\ell_2\subset S_2$, then
$
\sigma_*(d_1b_1)=d_1w_2,
$
and the gluing condition would force
$
d_1w_2=d_2b_2$ in $\CH^1(Q_2).
$
Since $b_2$ and $w_2$ are independent, this implies
$d_1=d_2=0$, contradicting the assumption $d_i\ge 1$.
Hence this case is impossible.
This proves the classification of all possible cases.
For the final claim, if $\ell_1\subset S_1$, then
\[
\sigma_*\bigl([\widetilde S_1\cap Q_1]\bigr)
=
\sigma_*\bigl((d_1-1)b_1+w_1\bigr)
=
(d_1-1)w_2+b_2.
\]
Its degree with respect to the ruling $Q_2\to \ell_2$ is
\[
\bigl((d_1-1)w_2+b_2\bigr)\cdot b_2
=
(d_1-1)\,w_2\cdot b_2+b_2^2
=
d_1-1.
\]
The statement for $\ell_2\subset S_2$ is symmetric.
\end{proof}

\subsection{Fulton's specialisation map for the semistable smoothing}
\label{subsec:chow-specialisation}
%======================================================================

This subsection relates the intersection theory of the singular central fibre
\[
Z_0=\widetilde Z=\widetilde Z_1\cup_Q \widetilde Z_2
\]
to the geometry of the nearby smooth fibres.
The specialisation map is defined on the total space of the family.
Since the special fibre of the local equation $t=uv$ is the sum of the
two Cartier divisors $u=0$ and $v=0$, the specialised cycle decomposes
accordingly.

Let $\Delta\subset \C$ be a small disc with coordinate $t$, and let
$
\pi:\mathcal Z\longrightarrow \Delta
$
be a flat projective morphism with smooth total space and central fibre
\[
Z_0:=\pi^{-1}(0)=\widetilde Z=\widetilde Z_1\cup_Q \widetilde Z_2,
\]
where  $\widetilde Z_1,\widetilde Z_2$ are smooth, $Q=\widetilde Z_1\cap \widetilde Z_2$
is smooth, and analytically locally along $Q$ there exist coordinates
$(u,v,z_3,\dots,z_n)$ on $\mathcal Z$ such that
\begin{equation*}\label{eq:semistable-model}
t=uv,\qquad
\widetilde Z_1=\{u=0\},\qquad
\widetilde Z_2=\{v=0\},\qquad
Q=\{u=v=0\}.
\end{equation*}

Passing to the formal neighbourhood of the origin, set
\begin{equation*}\label{eq:RK-def}
R:=\C[[t]],\qquad K:=\C((t)),\qquad
\mathcal Z_R:=\mathcal Z\times_\Delta \Spec R.
\end{equation*}
Then $\mathcal Z_R$ is flat and projective over $\Spec R$, its generic
fibre is
\[
(\mathcal Z_R)_K:=\mathcal Z_R\times_{\Spec R}\Spec K,
\]
and its special fibre is canonically identified with $Z_0$.

Throughout this subsection, $\CH_r(X)$ denotes the Chow group of $r$-dimensional cycles modulo rational equivalence.

\begin{proposition}\label{prop:specialisation-chow}
For every $r\in \Z$ there is a canonical specialisation homomorphism
\begin{equation*}\label{eq:sp-def}
\operatorname{sp}:\CH_r\bigl((\mathcal Z_R)_K\bigr)
\longrightarrow \CH_r(Z_0).
\end{equation*}
If $V_K\subset (\mathcal Z_R)_K$ is an integral closed subscheme of
dimension $r$ and $\overline V\subset \mathcal Z_R$ is its
scheme-theoretic closure, then
\begin{equation*}\label{eq:sp-formula}
\operatorname{sp}\bigl([V_K]\bigr)=i_0^!\bigl([\overline V]\bigr),
\end{equation*}
where
$
i_0:Z_0\hookrightarrow \mathcal Z_R
$
is the inclusion of the special fibre, viewed as an effective Cartier
divisor cut out by $t$.
If $\overline V$ meets $Z_0$ properly, then
\begin{equation*}\label{eq:sp-formula-proper}
\operatorname{sp}\bigl([V_K]\bigr)=\bigl[\overline V\cap Z_0\bigr].
\end{equation*}
\end{proposition}

\begin{proof}
This is Fulton's specialisation map for a family over a
discrete valuation ring, applied to the Cartier divisor
$Z_0\subset \mathcal Z_R$. When $\overline V$ meets $Z_0$ properly,
the refined Gysin pull-back agrees with the class of the
scheme-theoretic intersection.
\end{proof}
Refined Gysin pull-backs are taken in Fulton's sense. For a regular closed
immersion $f:Y\hookrightarrow X$ of codimension $c$,
$f^!: \CH_r(X)\to\CH_{r-c}(Y)$ agrees with scheme-theoretic
intersection when the intersection is proper and remains defined in
the presence of excess intersection. In particular, $i_0^!$ gives
specialisation to the central fibre and $i_i^!,j_i^!$ give the
component and double-locus traces; see
\cite[\S6.2, \S8.1, \S20.3]{Fulton}.
\begin{definition}\label{rem:proper-position}
An integral cycle $V_K\subset(\mathcal Z_R)_K$ is in \emph{strong
proper position} if its closure $\overline V\subset\mathcal Z_R$ has
no irreducible component contained in $\widetilde Z_1$,
$\widetilde Z_2$, or $Q$, and no irreducible component of
$\overline V\cap\widetilde Z_i$ is contained in $Q$. Equivalently,
$\overline V$ meets $\widetilde Z_1$, $\widetilde Z_2$, and $Q$
properly. In this case the relevant refined pull-backs are represented
by the corresponding scheme-theoretic intersections.
\end{definition}
For $i=1,2$, write
$
i_i:\widetilde Z_i\hookrightarrow \mathcal Z_R $
for the closed immersions of the two components of the special fibre,
and
$
j_i:Q\hookrightarrow \widetilde Z_i
$
for the inclusions of the double locus.

\begin{proposition} 
\label{prop:sp-decomp}
Let $V_K\subset(\mathcal Z_R)_K$ be an integral closed subscheme of
dimension $r$, with closure $\overline V\subset\mathcal Z_R$. Set
\begin{equation*}\label{eq:alpha-i-def}
\alpha_i:=i_i^!\bigl([\overline V]\bigr)\in\CH_r(\widetilde Z_i),
\qquad i=1,2.
\end{equation*}
Then:
\begin{enumerate}
\item%\label{item:sp-decomp-sum}
\begin{equation}\label{eq:sp-sum-general}
\operatorname{sp}\bigl([V_K]\bigr)
=\phi_{1*}\alpha_1+\phi_{2*}\alpha_2
\qquad\text{in }\CH_r(Z_0).
\end{equation}

\item%\label{item:sp-decomp-match}
The refined restrictions of $\alpha_1$ and $\alpha_2$ to the double
locus coincide:
\begin{equation}\label{eq:sp-matching-general}
j_1^!\alpha_1=j_2^!\alpha_2
\qquad\text{in }\CH_{r-1}(Q),
\end{equation}
equivalently
\begin{equation}\label{eq:sp-matching-sigma-general}
j_1^!\alpha_1=\sigma^*\bigl(j_2^!\alpha_2\bigr)
\qquad\text{on }Q_1.
\end{equation}
\end{enumerate}
\end{proposition}

\begin{proof}
Since $Z_0=\operatorname{div}(t)$ and locally $t=uv$, one has an
equality of effective Cartier divisors
\begin{equation*}\label{eq:cartier-split}
Z_0=\widetilde Z_1+\widetilde Z_2
\qquad\text{in }\mathcal Z_R.
\end{equation*}
By additivity of refined Gysin pull-back for Cartier divisors
\cite[\S2.3]{Fulton},
\[
i_0^!\bigl([\overline V]\bigr)
=\phi_{1*}i_1^!\bigl([\overline V]\bigr)
+\phi_{2*}i_2^!\bigl([\overline V]\bigr),
\]
which is \eqref{eq:sp-sum-general}.
Now, 
for \eqref{eq:sp-matching-general}, let $k:Q\hookrightarrow\mathcal Z_R$
be the inclusion. Since the total space $\mathcal Z_R$ is regular and
$Q=\widetilde Z_1\cap\widetilde Z_2$ is the scheme-theoretic
intersection of two Cartier divisors on $\mathcal Z_R$, the inclusion
$k$ is a regular immersion of codimension $2$ globally on
$\mathcal Z_R$ (not only in the local semistable chart). Moreover,
$k=i_1\circ j_1=i_2\circ j_2$ as maps into $\mathcal Z_R$.
By functoriality of refined Gysin homomorphisms for compositions of
regular immersions \cite[\S6.2]{Fulton},
\[
k^!=j_1^!\circ i_1^!=j_2^!\circ i_2^!.
\]
Applying this to $[\overline V]$ gives $j_1^!\alpha_1=j_2^!\alpha_2$.
The reformulation \eqref{eq:sp-matching-sigma-general} follows by
transport through $\sigma$.
\end{proof}

\begin{corollary}
\label{cor:sp-by-components}
Under the hypotheses of \Cref{prop:sp-decomp}, assume
moreover that $V_K$ is in strong proper position
(\Cref{rem:proper-position}). Set
\begin{equation*}\label{eq:Vi-VQ-def}
V_i:=\overline V\cap\widetilde Z_i\subset\widetilde Z_i,
\qquad
V_Q:=\overline V\cap Q\subset Q.
\end{equation*}
Then $\alpha_i=[V_i]$, so
\begin{equation}\label{eq:sp-sum}
\operatorname{sp}\bigl([V_K]\bigr)=\phi_{1*}[V_1]+\phi_{2*}[V_2],
\end{equation}
and
\begin{equation}\label{eq:sp-matching}
j_1^![V_1]=[V_Q]=j_2^![V_2]
\qquad\text{in }\CH_{r-1}(Q),
\end{equation}
equivalently $j_1^![V_1]=\sigma^*(j_2^![V_2])$ on $Q_1$.
\end{corollary}

\begin{proof}
Since $V_K$ is in strong proper position, $\overline V$ meets each
$\widetilde Z_i$ properly, so $\alpha_i=i_i^!([\overline V])=[V_i]$
by the proper case of refined intersection. This gives
\eqref{eq:sp-sum}. By the second clause of
\Cref{rem:proper-position}, each $V_i$ also meets $Q$ properly,
so $j_i^![V_i]=[V_Q]$, which is \eqref{eq:sp-matching}.
\end{proof}

Since $Z_0$ is singular, Chern classes on the central fibre are
understood operationally.

\begin{lemma}\label{lem:chern-commutes-sp}
Let $\mathcal E$ be a vector bundle on $\mathcal Z_R$. Then for every
$i\ge 0$ and every $\beta\in \CH_*\bigl((\mathcal Z_R)_K\bigr)$ one has
\begin{equation}\label{eq:chern-sp-commute}
\operatorname{sp}\Bigl(
c_i\bigl(\mathcal E|_{(\mathcal Z_R)_K}\bigr)\cap \beta
\Bigr)
=
c_i\bigl(\mathcal E|_{Z_0}\bigr)\cap \operatorname{sp}(\beta),
\end{equation}
where
$
c_i\bigl(\mathcal E|_{Z_0}\bigr)\in A^i(Z_0)
$
is the operational Chern class on the singular fibre.
\end{lemma}

\begin{proof}
By \Cref{prop:specialisation-chow},
$
\operatorname{sp}(\beta)=i_0^!(\overline\beta),
$
where $\overline\beta$ denotes the closure class on $\mathcal Z_R$.
Operational Chern classes commute with refined Gysin pull-backs.
Applying this to the Cartier divisor inclusion $i_0$ gives
\eqref{eq:chern-sp-commute}.
\end{proof}

Lemma~\ref{lem:chern-commutes-sp} gives the following lifting criterion for intersection classes defined by global line bundles on the family.

\begin{proposition}\label{prop:practical-lifting}
Let $\mathcal L_1,\dots,\mathcal L_m$ be line bundles on $\mathcal Z_R$.
Write
\begin{equation*}\label{eq:Lak-La0-def}
L_{a,K}:=\mathcal L_a|_{(\mathcal Z_R)_K},
\qquad
L_{a,0}:=\mathcal L_a|_{Z_0},
\qquad
L_{a,i}:=\mathcal L_a|_{\widetilde Z_i}.
\end{equation*}
Let $P\in\Z[t_1,\dots,t_m]$ be a polynomial. Then
\begin{equation}\label{eq:lifting-formula}
\operatorname{sp}\Bigl(
P\bigl(c_1(L_{1,K}),\dots,c_1(L_{m,K})\bigr)
\cap[(\mathcal Z_R)_K]
\Bigr)
=
P\bigl(c_1(L_{1,0}),\dots,c_1(L_{m,0})\bigr)\cap[Z_0].
\end{equation}
Moreover, if one denotes by
$
\alpha_i\in \CH_*(\widetilde Z_i)
$
the two component classes appearing in the componentwise decomposition
of the specialisation, then
\begin{equation}\label{eq:lifting-components}
\alpha_i=
P\bigl(c_1(L_{1,i}),\dots,c_1(L_{m,i})\bigr)\cap[\widetilde Z_i],
\qquad i=1,2,
\end{equation}
and these classes satisfy
\begin{equation}\label{eq:lifting-refined-match}
j_1^!\alpha_1=j_2^!\alpha_2
\qquad\text{in }\CH_*(Q).
\end{equation}
\end{proposition}

\begin{proof}
Apply \Cref{lem:chern-commutes-sp} repeatedly to the fundamental
class $[(\mathcal Z_R)_K]$. This gives \eqref{eq:lifting-formula}.

Now apply \Cref{prop:sp-decomp} to the specialised
cycle
\[
P\bigl(c_1(L_{1,K}),\dots,c_1(L_{m,K})\bigr)\cap[(\mathcal Z_R)_K].
\]
Its two component classes are the refined pull-backs of the global specialised class to $\widetilde Z_1$ and $\widetilde Z_2$, hence are the classes in \eqref{eq:lifting-components}.
The refined matching \eqref{eq:lifting-refined-match} is then
\eqref{eq:sp-matching-general} of \Cref{prop:sp-decomp}.
\end{proof}

The component classes in \Cref{prop:practical-lifting} match automatically on $Q$; in the bases $(z_i,w_i)$ this is expressed by the formulas \eqref{eq:sigma-on-w}--\eqref{eq:sigma-on-z}.

\section{Kato--Nakayama spaces and the fixed-phase neck}
\label{sec:KN}

The semistable equation $t=uv$ retains angular information that is lost under ordinary specialisation.

In the local model $t=uv$, a point of $Z_t$
near $Q$ is described by two complex numbers $u$ and $v$ with
$uv=t=|t|e^{i\theta}$. As $|t|\to 0$, the moduli $|u|$ and $|v|$
both go to zero, but the phases $\rho_1=u/|u|$ and $\rho_2=v/|v|$
survive and satisfy $\rho_1\rho_2=e^{i\theta}$. The
\emph{Kato--Nakayama space} $\mathcal Z^{\log}$ is the topological
space retaining these phases. Over the double locus $Q$, fixing the
phase of $t$ leaves a circle of compatible phase pairs at each point. The fixed-phase fibre
\begin{equation}\label{eq:KN-intro-fibre}
Q^{\log}\big|_\theta
:= \bigl\{(\rho_1,\rho_2)\in S^1\times S^1\mid \rho_1\rho_2
=e^{i\theta}\bigr\} \longrightarrow Q
\end{equation}
is a principal $S^1$-bundle over $Q$. Only its local description in semistable charts $uv=t$ near the double locus is required; no global identification of a chart with the entire Donaldson--Friedman family is used.

\subsection{Divisorial log structures and the Kato--Nakayama space}
\label{subsec:KN-two-divisors}

The logarithmic notation needed to describe the phase information in the semistable local model $t=uv$ is recalled briefly. Standard references are
\cite{KatoNakayama1999,Ogus2018}.

Let $X$ be a complex manifold and let
$
D=D_1\cup D_2\subset X
$
be a simple normal crossings divisor with two smooth components
meeting transversely. Set $U:=X\setminus D$ and let
$j:U\hookrightarrow X$ be the inclusion.

\begin{definition}%\label{def:divisorial-log}
The \emph{divisorial log structure} associated with $(X,D)$ is the
sheaf of commutative monoids
$
M_X := \OO_X\cap j_*\OO_U^\times \subset j_*\OO_U,
$
whose sections over an open $V\subset X$ are the holomorphic
functions on $V$ that are invertible away from $D$. The inclusion
$
\alpha:M_X\longrightarrow \OO_X
$
is the \emph{structure morphism}, and the triple $(X,M_X,\alpha)$ is
a \emph{log analytic space}. It satisfies
$\alpha^{-1}(\OO_X^\times)\cong\OO_X^\times$, meaning that
invertible functions are the same in $M_X$ and in $\OO_X$.
\end{definition}

The characteristic monoid $\overline M_{X,x}:=M_{X,x}/\OO_{X,x}^\times$
records the local orders of vanishing: it is $0$ on $U$, isomorphic to
$\NN$ on $D_i\setminus D_j$, and isomorphic to $\NN^2$ on
$D_1\cap D_2$, with one generator for each local branch.

\begin{definition}\label{def:KN}
The \emph{Kato--Nakayama space} $X^{\log}$ is the set of pairs
$(x,h)$, where $x\in X$ and
$
h:M_{X,x}^{\mathrm{gp}}\longrightarrow S^1
$
is a group homomorphism from the groupification of $M_{X,x}$
satisfying
\begin{equation}\label{eq:KN-unit-condition}
h(u)=\frac{u(x)}{|u(x)|}
\qquad\text{for every }u\in \OO_{X,x}^\times\subset M_{X,x}.
\end{equation}
Equipped with the topology induced by local charts of the log
structure, the natural projection
\begin{equation*}\label{eq:KN-projection}
\rho:X^{\log}\longrightarrow X,\qquad (x,h)\longmapsto x,
\end{equation*}
is continuous and proper, and restricts to a homeomorphism over $U$.
\end{definition}
Over $U$ the projection $X^{\log}\to X$ is a homeomorphism. Over $D$, each local branch contributes one circle recording its phase.

Choose holomorphic coordinates $(z_1,\dots,z_n)$ near $x\in X$ with
\begin{equation*}\label{eq:local-divisors}
D_1\cap V=\{z_1=0\},\qquad D_2\cap V=\{z_2=0\}.
\end{equation*}
Write $z_i=r_i e^{i\theta_i}$ with $r_i\in\R_{\ge 0}$ and
$\theta_i\in S^1$. Then there is a canonical homeomorphism
\begin{equation*}\label{eq:KN-local-chart}
V^{\log}\cong
\bigl\{(r_1,\theta_1,r_2,\theta_2,z_3,\dots,z_n)
\mid r_i\in\R_{\ge 0},\ \theta_i\in S^1\bigr\},
\end{equation*}
under which the projection $\rho$ is
\begin{equation*}\label{eq:KN-local-proj}
\rho(r_1,\theta_1,r_2,\theta_2,z_3,\dots,z_n)
=(r_1 e^{i\theta_1},\, r_2 e^{i\theta_2},\, z_3,\dots,z_n).
\end{equation*}
Thus $X^{\log}$ is locally the real oriented blow-up of $X$ along $D$:
each complex coordinate $z_i$ vanishing on $D_i$ is replaced by its
modulus $r_i$ and its phase $\theta_i$. Under a coordinate change
$z_i'=u_i z_i$ with $u_i$ a holomorphic unit, the phase transforms by
\begin{equation*}\label{eq:phase-transform}
\theta_i' = \frac{u_i}{|u_i|}\,\theta_i,
\end{equation*}
so the local models glue canonically; see
\cite[\S1--\S2]{KatoNakayama1999}.
Define the boundary strata
\begin{equation*}\label{eq:boundary-strata}
\partial_i X^{\log}:=\rho^{-1}(D_i),\qquad
\partial_{12}X^{\log}:=\rho^{-1}(D_1\cap D_2).
\end{equation*}
The fibres of $\rho$ are:
\begin{equation*}\label{eq:KN-fibres}
\rho^{-1}(x)\cong
\begin{cases}
\{\mathrm{pt}\} & x\in U, \\
S^1 & x\in D_i\setminus D_j, \\
S^1\times S^1 & x\in D_1\cap D_2.
\end{cases}
\end{equation*}
Over a point of $D_1\cap D_2$ the fibre is a two-torus, one circle for
each branch. Imposing $\rho_1\rho_2=e^{i\theta}$ leaves the circle in
\eqref{eq:KN-intro-fibre}.

The construction applies equally to a normal crossings space. If
\begin{equation*}\label{eq:NCC-union}
X_0=X_1\cup_Q X_2,
\end{equation*}
with $X_1,X_2$ smooth and $Q\subset X_i$ a smooth divisor identified
through a fixed isomorphism, then $X_0$ carries the fine saturated log
structure induced by the normal crossing locus. Near a point of $Q$,
in coordinates $\{uv=0\}\subset\C^n$, the space $X_0^{\log}$ is
described by the same local coordinates, with one phase variable for each branch.
This is the case needed for the Donaldson--Friedman central fibre
$\widetilde Z=\widetilde Z_1\cup_Q \widetilde Z_2$.

\begin{figure}[htbp]
    \centering
\includegraphics[width=0.7\textwidth]{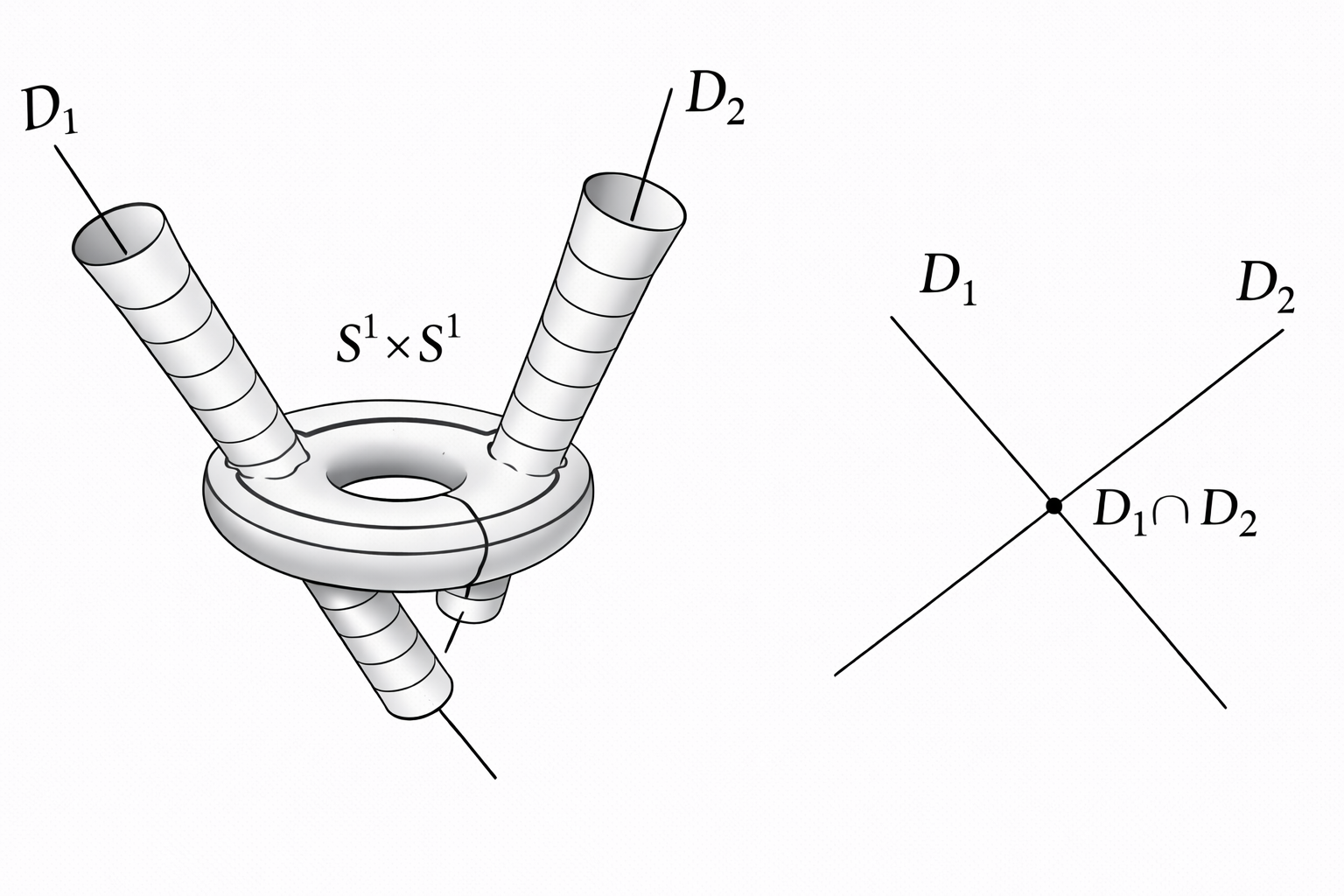}
   \caption{Local model of the Kato--Nakayama space near a normal crossing point.}
\end{figure}

\subsection{The fixed-phase neck and an auxiliary \texorpdfstring{$\RP^3$}{RP3}-construction}
\label{subsec:KN-real-lines}

The notation of Section~\ref{sec:CH-pushout} is retained: the SNC space
\(Z_0=\widetilde Z_1\cup_Q\widetilde Z_2\),
\(Q\simeq Q_1\simeq Q_2\),
with $\widetilde Z_i=\Bl_{\ell_i}(Z_i)$ and exceptional quadric
$Q\simeq \PP^1\times\PP^1$.
Fix the branch $Y_1:=\widetilde Z_1$ and consider the unit circle bundle of the normal line bundle of $Q$ in $Y_1$. The equation $uv=t$ is used only \emph{chartwise}, as a local analytic model near $Q$.

By a \emph{semistable chart adapted to $Y_1$ along $Q$} we mean a neighbourhood $V\subset Q$, an analytic space
$\mathcal V$, and a flat map
\[
\pi_V:\mathcal V\longrightarrow \Delta
\]
such that, analytically near $V$, the space $\mathcal V$ is described
by coordinates
$
(u,v,z_3,\dots,z_n,t)
$
with equation
\begin{equation*}\label{eq:semistable-chart-KN}
t=uv,
\qquad
\mathcal V_0=\{uv=0\},
\qquad
V=\{u=v=0\},
\end{equation*}
and such that the branch $\{u=0\}\subset\mathcal V_0$ is identified
with a neighbourhood of $V$ inside $Y_1$. Thus, inside the chosen
branch $Y_1$, the divisor $Q$ is locally cut out by the coordinate
$v$.
For $t\neq 0$, write
\[
\mathcal V_t:=\pi_V^{-1}(t),
\qquad
t=|t|e^{i\theta},
\]
and define the phase functions on the smooth fibre $\mathcal V_t=\{uv=t\}$:
\begin{equation*}\label{eq:phases-KN}
\rho_1:=u/|u|\in S^1,\qquad \rho_2:=v/|v|\in S^1,
\qquad \rho_1\rho_2=e^{i\theta}.
\end{equation*}
As $|t|\to 0$, the moduli $|u|,|v|\to 0$ but the phases survive,
subject to the single constraint $\rho_1\rho_2=e^{i\theta}$.

\begin{lemma}\label{lem:KN-neck-rewrite}
Let $V\subset Q$ be the base of a semistable chart adapted to $Y_1$,
and fix $\theta\in \R/2\pi\Z$.
\begin{enumerate}
\item\label{item:KN-neck-1}
The fibre $\bigl(\pi_V^{\log}\bigr)^{-1}(0,\theta)$ restricts over $V$
to a principal $S^1$-bundle
\[
P_{V,\theta}\longrightarrow V,
\]
whose fibre over a point of $V$ is
\begin{equation}\label{eq:fibre-circle-rewrite}
\bigl\{(\rho_1,\rho_2)\in S^1\times S^1
\mid \rho_1\rho_2=e^{i\theta}\bigr\}\cong S^1,
\end{equation}
with $S^1$ acting by
\begin{equation}\label{eq:S1-action-KN}
a\cdot(\rho_1,\rho_2)=(a\rho_1,a^{-1}\rho_2).
\end{equation}

\item\label{item:KN-neck-2}
For $t=|t|e^{i\theta}\neq 0$, define the equal-modulus neck
\begin{equation}\label{eq:neck-rewrite}
N_{V,t,\theta}:=
\bigl\{(u,v,q)\in \mathcal V_t \mid |u|=|v|=\sqrt{|t|}\bigr\}.
\end{equation}
Then $N_{V,t,\theta}\to V$ is a principal $S^1$-bundle.

\item\label{item:KN-neck-3}
The two $S^1$-bundles in \eqref{eq:fibre-circle-rewrite} and
\eqref{eq:neck-rewrite} are canonically isomorphic via
\begin{equation}\label{eq:identify-rewrite}
(u,v,q)\longmapsto \bigl(q;\,u/|u|,\,v/|v|\bigr).
\end{equation}
In particular, in every adapted semistable chart the fixed-phase Kato--Nakayama
circle is the topological limit of the equal-modulus neck as
$|t|\to 0$ along the ray $t=|t|e^{i\theta}$.
\end{enumerate}
\end{lemma}

\begin{proof}
\emph{Proof of \ref{item:KN-neck-1}.}
By \Cref{def:KN}, a point of
$\bigl(\pi_V^{\log}\bigr)^{-1}(0,\theta)$ above $x\in V$ is a
homomorphism
\[
h:M_{\mathcal V,x}^{\mathrm{gp}}\to S^1
\]
satisfying \eqref{eq:KN-unit-condition}. If $u$ and $v$ are the two
branch coordinates in the semistable chart, set
\[
\rho_1:=h(u),\qquad \rho_2:=h(v).
\]
Since $uv=t$ and the point of the base is $(0,\theta)\in\Delta^{\log}$,
one has
\[
\rho_1\rho_2=h(t)=e^{i\theta}.
\]
Thus the fibre is exactly \eqref{eq:fibre-circle-rewrite}, and the
action \eqref{eq:S1-action-KN} is free and transitive.
\emph{Proof of \ref{item:KN-neck-2}.}
In the local model $uv=t$, the condition $|u|=|v|$ together with
$|u||v|=|t|$ is equivalent to
\[
|u|=|v|=\sqrt{|t|}.
\]
Hence $N_{V,t,\theta}$ is precisely the locus of equal moduli. Over a
fixed point of $V$, the remaining freedom is the phase pair
\[
(\rho_1,\rho_2)=\bigl(u/|u|,\ v/|v|\bigr)\in S^1\times S^1
\]
subject to $\rho_1\rho_2=e^{i\theta}$, so the fibre is $S^1$.
\emph{Proof of \ref{item:KN-neck-3}.}
The map \eqref{eq:identify-rewrite} is well defined because on
$N_{V,t,\theta}$ one has
\[
\frac{u}{|u|}\frac{v}{|v|}=\frac{uv}{|uv|}=e^{i\theta}.
\]
It is fibrewise $S^1$-equivariant for the action
\eqref{eq:S1-action-KN}, and in local coordinates it is plainly a
bundle isomorphism.
\end{proof}

Fixing the branch $Y_1=\widetilde Z_1$, the local circles glue to a global circle bundle.

\begin{lemma}\label{lem:normal-exceptional}
For each branch $Y_i=\widetilde Z_i$, the normal bundle of the
exceptional divisor $Q\subset Y_i$ satisfies
$
\mathcal N_{Q/Y_i}\simeq \OO_Q(-1).
$
\end{lemma}

\begin{proof}
Fix $i\in\{1,2\}$. Since $Q\subset Y_i=\widetilde Z_i$ is a Cartier
divisor, one has
$
\mathcal N_{Q/Y_i}\simeq\OO_{Y_i}(Q)|_Q.
$
The Proj description of the blow-up gives
$
\OO_{Y_i}(-Q)|_Q\simeq\OO_Q(1),
$
hence
$
\OO_{Y_i}(Q)|_Q\simeq\OO_Q(-1).
$
Therefore
\[
\mathcal N_{Q/Y_i}\simeq \OO_Q(-1).
\]
See \cite[II.7]{Hartshorne} or \cite[Appendix~B]{Fulton}.
\end{proof}

\begin{lemma}\label{lem:S1-associated-line}
Let $P\to B$ be a principal $S^1$-bundle over a paracompact space,
and let $L_P:=P\times_{S^1}\C$ be the associated complex line bundle.
Then $c_1(P)=c_1(L_P)\in H^2(B;\Z)$, and after choosing a Hermitian
metric on $L_P$ the unit circle bundle $S(L_P)\to B$ satisfies
\begin{equation*}\label{eq:circle-bundle-iso}
P\simeq S(L_P)
\end{equation*}
$S^1$-equivariantly, canonically up to homotopy.
\end{lemma}

\begin{proof}
Both $P$ and $L_P$ are glued by the same $S^1$-valued cocycle, and
$c_1$ is the connecting class in the exponential sequence; see
\cite[Ch.~4]{Husemoller} or \cite[\S14]{MilnorStasheff}.
\end{proof}

\begin{lemma}\label{lem:O1-restrict-fibre}
With $N_{\ell/Z}\simeq\OO_{\PP^1}(1)^{\oplus 2}$ and
$Q=\PP(N_{\ell/Z})$, for every fibre $F=g^{-1}(x)\simeq\PP^1$ of
$g:Q\to\ell$ one has
\begin{equation*}\label{eq:O1-restrict}
\OO_Q(1)\big|_F\simeq \OO_{\PP^1}(1).
\end{equation*}
\end{lemma}

\begin{proof}
The fibre is $F=\PP((N_{\ell/Z})_x)\simeq\PP^1$, and $\OO_Q(1)$
restricts to the tautological quotient bundle on that fibre.
\end{proof}

Write $Y_i:=\widetilde Z_i$. By \Cref{lem:normal-exceptional},
for each branch $Y_i$ one has
\begin{equation*}\label{eq:normal-bundle-recall}
Q\subset Y_i,\ 
\mathcal N_{Q/Y_i}\simeq\OO_Q(-1),\ 
Q=\PP(N_{\ell/Z})\simeq\PP^1\times\PP^1.
\end{equation*}
All branchwise constructions in this subsection use $Y_1$. The convention
\[\PP(E)=\Proj(\Sym E)\]
is used throughout.

\begin{lemma}\label{lem:KN-is-normal-circle}
Fix a phase $\theta\in \R/2\pi\Z$. The local bundles
$P_{V,\theta}\to V$ from \Cref{lem:KN-neck-rewrite}, constructed
on semistable charts adapted to $Y_1$, glue canonically to a global
principal $S^1$-bundle over $Q$. This global bundle is canonically
isomorphic, up to isotopy, to the unit circle bundle
\[
S(\mathcal N_{Q/Y_1})\longrightarrow Q.
\]
It is denoted by
$
Q^{\log}\big|_\theta\longrightarrow Q.
$
In particular,
\begin{equation*}\label{eq:c1-KN-normal-rewrite}
c_1\!\left(Q^{\log}\big|_\theta\right)
=c_1(\mathcal N_{Q/Y_1})\in H^2(Q;\Z).
\end{equation*}
\end{lemma}

\begin{proof}
Since $Q \simeq \PP^1 \times \PP^1$ is compact and the semistable local
model $t = uv$ holds analytically near every point of $Q$ in the
Donaldson--Friedman family, we can cover $Q$ by finitely many open sets
$\{V_\alpha\}$, each equipped with a semistable chart adapted to the
branch $Y_1$. It suffices to work in these charts and verify that the
local constructions glue consistently.
In each chart $V_\alpha$, let $v_\alpha$ be the local coordinate cutting
out $Q$ inside the branch $Y_1$. By
\Cref{lem:KN-neck-rewrite}, the fibre of the Kato--Nakayama
phase bundle over $V_\alpha$ is a principal $S^1$-bundle, and the
relevant local datum is simply the phase of $v_\alpha$:
\[
\eta_\alpha := \frac{v_\alpha}{|v_\alpha|} \in S^1.
\]
Once $\eta_\alpha$ is known, the phase of the other branch is determined
by the constraint $\rho_{1,\alpha} \cdot \eta_\alpha = e^{i\theta}$,
namely $\rho_{1,\alpha} = e^{i\theta}\eta_\alpha^{-1}$.
On an overlap $V_\alpha \cap V_\beta$, the two local defining equations
of $Q$ inside $Y_1$ are related by a nowhere-vanishing holomorphic
function $a_{\alpha\beta} \in \mathcal{O}_Q^\times$:
\[
v_\alpha = a_{\alpha\beta} \, v_\beta.
\]
The local phase coordinate therefore transforms as
\[
\eta_\alpha = \frac{a_{\alpha\beta}}{|a_{\alpha\beta}|}\,\eta_\beta.
\]
This is precisely the transition rule for the unit circle bundle of the
line bundle $\mathcal{N}_{Q/Y_1}$, whose cocycle is
$(a_{\alpha\beta})$. Hence the local phase bundles defined in each chart
glue to a globally defined principal $S^1$-bundle, canonically
isomorphic to $S(\mathcal{N}_{Q/Y_1}) \to Q$. The equality of first
Chern classes then follows from \Cref{lem:S1-associated-line}.
\end{proof}

The identification in \Cref{lem:KN-is-normal-circle} is branchwise: after choosing $Y_1=\widetilde Z_1$, one has $Q^{\log}|_\theta\simeq S(\mathcal N_{Q/Y_1})$. Only the local semistable equation $uv=t$ near $Q$ is used.

\begin{proposition}\label{prop:KN-degree-fibre}
Identify $Y_1$ with $\widetilde Z_1$. For a ruling fibre
$F\simeq\PP^1$ of $g:Q\to\ell$, the restriction
\[
\left(Q^{\log}\big|_\theta\right)\big|_F\to F
\]
is a principal $S^1$-bundle with
\begin{equation*}\label{eq:c1-KN-fibre}
\left|\,c_1\!\left(\left(Q^{\log}\big|_\theta\right)\big|_F\right)\right|
=1
\qquad\text{in }H^2(F;\Z)\cong\Z.
\end{equation*}
In particular, its total space is diffeomorphic to $S^3$.
\end{proposition}

\begin{proof}
By \Cref{lem:KN-is-normal-circle} and
\Cref{lem:normal-exceptional}, we obtain
$
Q^{\log}|_\theta\simeq S(\OO_Q(-1)).
$
Restricting to $F$ via \Cref{lem:O1-restrict-fibre} gives
\[
\OO_Q(-1)|_F\simeq\OO_{\PP^1}(-1).
\]
Therefore the associated circle bundle has first Chern class equal to
$\pm1$, depending on the orientation convention for the $S^1$-fibres.
Hence its total space is the Hopf bundle up to orientation, and in
particular is diffeomorphic to $S^3$.
\end{proof}

\begin{lemma}\label{lem:chern-character}
Let $P_0\to B$ be a principal $T^2=S^1\times S^1$-bundle with Chern
vector $(c_1^{(1)},c_1^{(2)})\in H^2(B;\Z)^2$. For the character
\begin{equation*}\label{eq:character-def}
\chi_{a,b}:T^2\to S^1,\qquad
\chi_{a,b}(\lambda_1,\lambda_2)=\lambda_1^a\lambda_2^b,
\end{equation*}
the associated $S^1$-bundle
$P:=P_0\times_{T^2,\chi_{a,b}}S^1$ satisfies
\begin{equation*}\label{eq:chern-character-formula}
c_1(P)=a\,c_1^{(1)}+b\,c_1^{(2)}\in H^2(B;\Z).
\end{equation*}
\end{lemma}

\begin{proof}
The associated $S^1$-bundle $P$ is obtained from $P_0$ by changing
structure group along the character $\chi_{a,b}$. Its transition
cocycles are obtained from those of $P_0$ by applying $\chi_{a,b}$,
which on the level of $U(1)$-valued cocycles gives
$(g_{\alpha\beta}^{(1)},g_{\alpha\beta}^{(2)})\mapsto
(g_{\alpha\beta}^{(1)})^a(g_{\alpha\beta}^{(2)})^b$.
By the additivity of the connecting homomorphism in the exponential
sequence $$0\to\Z\to\R\to S^1\to 0,$$ the first Chern class of $P$
is $a\,c_1^{(1)}+b\,c_1^{(2)}$; see \cite[Ch.~4]{Husemoller} or
\cite[\S14]{MilnorStasheff}.
\end{proof}

\begin{remark}
\label{rem:S3-vs-RP3}
It is important to distinguish two different geometric objects.

\smallskip
\noindent\emph{The $4$-manifold neck.}
The connected sum $X_1\#X_2$ is formed by removing a small open
$4$-ball near each $x_i$ and gluing the resulting boundaries
$S^3\simeq \partial B_i$. Thus the neck in the underlying
$4$-manifold is a genuine $3$-sphere.

\smallskip
\noindent\emph{The auxiliary \texorpdfstring{$\RP^3$}{RP3} in the local topology of the twistor degeneration.}
Over a ruling fibre $F\simeq S^2$, the fixed-phase circle bundle associated with $Y_1$ $Q^{\log}|_\theta|_F\to F$ is a Hopf bundle, hence has total
space $S^3$. Combining it with a second Hopf bundle over the same base
and quotienting by the anti-diagonal circle produces, by
\Cref{prop:antidiag-quotient-RP3}, a principal $S^1$-bundle
over $F$ with Chern class $2$, hence total space $\RP^3$.

This $\RP^3$ is an auxiliary topological object naturally attached to
the local twistor neck over the ruling fibre. It should \emph{not} be
identified with the $S^3$ along which the connected sum of the base
$4$-manifolds is performed.

\smallskip
\noindent
\Cref{prop:real-structure-Q} gives an independent algebraic copy of $\RP^3$ as the real locus inside $\PP H^0(Q,\OO_Q(1,1))$. No canonical identification between these two a priori different $\RP^3$-manifolds is asserted.
\end{remark}

\begin{proposition}\label{prop:antidiag-quotient-RP3}
Let $F\simeq S^2$ be a ruling fibre of $Q\to\ell$, and orient the circle fibres so that $S^1_{\mathrm{KN}}:=(Q^{\log}|_\theta)|_F\to F$ has $c_1=1$. Form
the principal $T^2$-bundle
\begin{equation*}\label{eq:T2-bundle}
P_0:=S^1_{\mathrm{KN}}\times_F S^3\longrightarrow F,
\end{equation*}
where $S^3\to F$ is the Hopf fibration (also of $c_1=1$), so the
Chern vector of $P_0$ is $(1,1)$. Let
$AS^1:=\{(\lambda,\lambda^{-1})\mid\lambda\in S^1\}\subset T^2$ be
the anti-diagonal subgroup, and set $P:=P_0/AS^1$. Then $P\to F$ is
a principal $S^1$-bundle with
$
c_1(P)=2,
$
and its total space is diffeomorphic to $L(2,1)\simeq\RP^3$.
\end{proposition}

\begin{proof}
The quotient map $q:T^2\to S^1$, $q(\lambda_1,\lambda_2)=\lambda_1\lambda_2$,
has kernel $AS^1$, so $P=P_0\times_{T^2,q}S^1$. This is the
character $(a,b)=(1,1)$, so \Cref{lem:chern-character} gives
$c_1(P)=1+1=2$. A principal $S^1$-bundle over $S^2$ with $c_1=2$
has total space $L(2,1)\simeq\RP^3$; see \cite[Ch.~4]{Husemoller}.
\end{proof}

\begin{lemma}
\label{prop:strong-def-retract-fixed-theta-rewrite}
Let $V\subset Q$ be the base of a semistable chart adapted to $Y_1$,
fix $\theta\in S^1$, and let $0<r\ll 1$. In the local semistable
model $t=uv$, the family of equal-modulus necks
\begin{equation*}\label{eq:Wrtheta-def}
N_{V,\rho,\theta}
:=\bigl\{(u,v,q)\in\mathcal V_{\rho e^{i\theta}}
\mid |u|=|v|=\sqrt{\rho}\bigr\},
\qquad 0<\rho\le r,
\end{equation*}
fits together with the local fixed-phase circle $P_{V,\theta}\to V$
at $\rho=0$ into a topological principal $S^1$-bundle over
$[0,r]\times V$. In particular, each neck $N_{V,\rho,\theta}$ is
canonically homeomorphic to the local fixed-phase circle $P_{V,\theta}$,
and this homeomorphism is compatible with the $S^1$-action and varies
continuously in $\rho$.
In particular, in the local model, the space
\begin{equation*}\label{eq:Wrtheta-def2}
W_{V,r,\theta}:=\bigsqcup_{0\le\rho\le r}N_{V,\rho,\theta}
\end{equation*}
deformation-retracts onto $P_{V,\theta}$ via the map
$(\rho,q,\eta)\mapsto(0,q,\eta)$, where $\eta:=v/|v|\in S^1$ is the
phase coordinate of \eqref{eq:identify-rewrite}.
\end{lemma}

\begin{proof}
The phase coordinate $\eta:=v/|v|\in S^1$ trivialises the family
globally over $[0,r]\times V$. Explicitly, each point of
$N_{V,\rho,\theta}$ for $\rho>0$ is parametrised as
\[
u=\sqrt{\rho}\,e^{i\theta}\eta^{-1},\quad
v=\sqrt{\rho}\,\eta,\quad q\in V,
\]
which satisfies $uv=\rho e^{i\theta}$ and $|u|=|v|=\sqrt{\rho}$.
At $\rho=0$ the same parameter $\eta$ determines the local phase pair
\[
\bigl(e^{i\theta}\eta^{-1},\,\eta\bigr)
\in
\bigl\{(\rho_1,\rho_2)\in S^1\times S^1\mid
\rho_1\rho_2=e^{i\theta}\bigr\}=P_{V,\theta}.
\]
This gives a continuous family over $[0,r]\times V$, and the
$S^1$-action on the $\eta$-coordinate defines the principal bundle
structure. The deformation retraction
\[
H(\lambda;\rho,q,\eta)
:=\bigl((1-\lambda)\rho,\,q,\,\eta\bigr),
\qquad\lambda\in[0,1],
\]
is continuous, $S^1$-equivariant, satisfies $H(0,-)=\mathrm{id}$
and $H(1,-)$ collapses $W_{V,r,\theta}$ onto $P_{V,\theta}$, and
fixes $P_{V,\theta}$ pointwise.
\end{proof}

\Cref{prop:strong-def-retract-fixed-theta-rewrite} is a local
statement, valid in each semistable chart adapted to $Y_1$. The
individual local retractions are compatible on overlaps because the
circle coordinate $\eta=v/|v|$ transforms by
$\eta\mapsto(a/|a|)\eta$ under $v\mapsto av$ with $a\in\OO_Q^\times$,
consistently with the cocycle of $S(\mathcal N_{Q/Y_1})$ identified in
\Cref{lem:KN-is-normal-circle}. Hence the local retractions glue
to a deformation retraction of the global neck family onto
$Q^{\log}|_\theta=S(\mathcal N_{Q/Y_1})$.

Thus the fixed-phase circle describes the limiting topology of the neck as $\rho\to0$. No strong deformation retraction of the full Kato--Nakayama space $\mathcal Z^{\log}$ is asserted; such a statement would require a separate analysis of the log structure on $\mathcal Z$.

\begin{theorem}\label{thm:DF-diagram-KN-Clemens}
Fix a phase $\theta\in S^1$ and consider the circle bundle associated with the chosen branch $Y_1$,
\[
Q^{\log}\big|_\theta:=S(\mathcal N_{Q/Y_1})\longrightarrow Q.
\]
Let $F\simeq \PP^1$ be a ruling fibre of $Q\to\ell$.
\begin{enumerate}
\item
The restriction
\[
\left(Q^{\log}\big|_\theta\right)\big|_F \longrightarrow F
\]
is the Hopf circle bundle up to orientation. In particular, its total
space is diffeomorphic to $S^3$.

\item
Let
\[
P_0:=\left(\left(Q^{\log}\big|_\theta\right)\big|_F\right)\times_F S^3
\longrightarrow F,
\]
where $S^3\to F$ is the Hopf fibration. If
\[
AS^1:=\{(\lambda,\lambda^{-1})\mid \lambda\in S^1\}\subset T^2,
\]
then the quotient
$
P:=P_0/AS^1
$
is a principal $S^1$-bundle over $F$ with first Chern class $2$.
Hence its total space is diffeomorphic to
$
L(2,1)\simeq \RP^3.$

\item
The anti-diagonal quotient therefore gives, over each ruling fibre $F$, an auxiliary $3$-manifold diffeomorphic to $\RP^3$. This $\RP^3$ is not the $S^3$
along which the connected sum of the underlying $4$-manifolds is
performed.
\end{enumerate}
\end{theorem}

\begin{proof}
Part (1) is \Cref{prop:KN-degree-fibre}. Part (2) is
\Cref{prop:antidiag-quotient-RP3}. Part (3) is the
interpretation recorded in \Cref{rem:S3-vs-RP3}.
Part (3) depends only on the normal circle bundle $Q^{\log}|_\theta$ and the Hopf fibration over $F$. The anti-diagonal quotient has first Chern class $2$, since the two factors have Chern classes $(1,1)$.
\end{proof}

%======================================================================
\subsection{The fixed-phase circle bundle over curves in \texorpdfstring{$Q$}{Q}}

Let $S_K\subset(\mathcal Z_R)_K$ be an integral surface whose closure
$\overline S$ is in strong proper position, and set
\[
S_i:=\overline S\cap\widetilde Z_i,\qquad c:=\overline S\cap Q.
\]
By \Cref{cor:sp-by-components}, the two traces agree on $Q$:
\[
j_1^![S_1]=[c]=j_2^![S_2].
\]
The class $[c]\in\CH_1(Q)$ determines the algebraic trace. The restriction of the fixed-phase Kato--Nakayama circle bundle $Q^{\log}|_\theta\to Q$ to $c$ contains the corresponding angular information.

\begin{proposition}\label{prop:KN-circle-on-c}
Let $c\subset Q$ be an irreducible reduced curve, and let
\[
Q^{\log}\big|_\theta\to Q
\]
be the fixed-phase circle bundle associated with $\widetilde Z_1$ and defined in
\Cref{lem:KN-is-normal-circle}. Set
\begin{equation*}\label{eq:Ptheta-def}
P_\theta:=\left(Q^{\log}\big|_\theta\right)\big|_c\longrightarrow c.
\end{equation*}
Then the associated complex line bundle satisfies
\begin{equation}\label{eq:Ltheta-normal}
L_\theta\simeq \mathcal N_{Q/\widetilde Z_1}\big|_c,
\end{equation}
and
\begin{equation*}\label{eq:c1-Ptheta}
c_1(P_\theta)=c_1(L_\theta)\in H^2(c;\Z).
\end{equation*}
\end{proposition}

\begin{proof}
By \Cref{lem:KN-is-normal-circle}, $Q^{\log}|_\theta$ is the
unit circle bundle of $\mathcal N_{Q/\widetilde Z_1}\simeq\OO_Q(-1)$
(\Cref{lem:normal-exceptional}). Restricting to $c$ gives
\eqref{eq:Ltheta-normal}, and equality of first Chern classes for a
principal $S^1$-bundle and its associated line bundle is
\Cref{lem:S1-associated-line}.
\end{proof}

The Chow class $[c]$
tells us \emph{where} the limiting surface meets $Q$, but it forgets
the phase. In a local semistable chart $uv=t$ near a point of $Q$,
as $t\to 0$ along a ray $t=|t|e^{i\theta}$, each point of the trace
acquires a compatible phase pair $(\rho_1,\rho_2)$ with
$\rho_1\rho_2=e^{i\theta}$. This information is represented by the restricted principal $S^1$-bundle $P_\theta\to c$ and is not detected by the Chow class.
The characteristic class of this bundle is
$c_1(\OO_Q(-1)|_c)$, which depends on the embedding of $c$ in $Q$.
If $c$ has bidegree $(a,b)$ on $Q\simeq \PP^1\times\PP^1$, then
\[
\deg \OO_Q(-1)|_c = -(a+b),
\]
so this class is nonzero for every nontrivial effective curve
$c\subset Q$.

%======================================================================
\subsection{Twistor geometry of \texorpdfstring{$Q$}{Q}
and its real structure}
%\label{subsec:twistor-Q}

The exceptional quadric $Q\simeq\PP^1\times\PP^1$ carries a natural real structure induced by the quaternionic structures of the twistor spaces $Z_i$. This subsection describes the induced real structure on $|\OO_Q(1,1)|$ and exhibits a concrete $\tau$-invariant real pencil with two non-real base points.

Let $\OO_Q(1,1)$ denote the Segre polarisation on
$Q\simeq\PP^1\times\PP^1$.

\begin{notation}
In this subsection $\tau:Q\to Q$ denotes the antiholomorphic involution of $Q$; the symbol $\sigma:Q_1\xrightarrow{\sim}Q_2$ continues to denote the biholomorphism exchanging the two rulings.
\end{notation}

\begin{proposition}\label{prop:real-structure-Q}
Let $Q\simeq\PP^1\times\PP^1$, and consider the antiholomorphic
involution
\begin{equation}\label{eq:tau-def}
\tau:\PP^1\times\PP^1\longrightarrow\PP^1\times\PP^1,\qquad
\bigl([z_0:z_1],[w_0:w_1]\bigr)\longmapsto
\bigl([-\overline z_1:\overline z_0],
[-\overline w_1:\overline w_0]\bigr).
\end{equation}
Then $\tau$ has empty real locus $Q^\tau=\varnothing$.
The induced antilinear involution on
$
H^0\!\bigl(Q,\OO_Q(1,1)\bigr)
$
has fixed real subspace of real dimension $4$, hence the fixed locus
in the projective space is naturally identified with
\begin{equation*}\label{eq:RP3-fixed}
\PP H^0\!\bigl(Q,\OO_Q(1,1)\bigr)^{\tau}\;\cong\;\RP^3.
\end{equation*}
This $\RP^3$ parametrises the $\tau$-invariant divisors in the linear
system $|\OO_Q(1,1)|$.
In coordinates $[z_0:z_1]\times[w_0:w_1]$ on $Q$, the embedding of
$\RP^3$ in $\PP^3$ via the basis
$z_0w_0,\,z_0w_1,\,z_1w_0,\,z_1w_1$ of $H^0(Q,\OO_Q(1,1))$ is
\begin{equation}\label{eq:RP3-embedding}
[a_1:a_2:b_1:b_2]\in\RP^3
\longmapsto
\bigl[a_1+ia_2\,:\,b_1+ib_2\,:\,-b_1+ib_2\,:\,a_1-ia_2\bigr]
\in\PP^3.
\end{equation}
Moreover, the sections
\begin{equation}\label{eq:s1s2-def}
s_1:=z_0w_0+z_1w_1,\qquad s_2:=z_0w_1-z_1w_0
\end{equation}
pan a $\tau$-invariant real pencil.
As one explicit example, this pencil defines a $\tau$-equivariant
rational map
$
h:Q\dashrightarrow\PP^1
$
satisfying
$
h\circ\tau=\overline{h}.
$
\end{proposition}

\begin{proof}
\emph{Step 1: the chosen real structure.}
Consider the antiholomorphic involution $\tau$ defined in
\eqref{eq:tau-def}. A point
$
\bigl([z_0:z_1],[w_0:w_1]\bigr)\in \PP^1\times\PP^1
$
is fixed by $\tau$ if and only if
\[
[z_0:z_1]=[-\overline z_1:\overline z_0],
\qquad
[w_0:w_1]=[-\overline w_1:\overline w_0].
\]
But the antiholomorphic involution
$
[z_0:z_1]\longmapsto [-\overline z_1:\overline z_0]
$
on $\PP^1$ has no fixed points, hence $\tau$ has no fixed points on
$Q$ and therefore $Q^\tau=\varnothing$.

\smallskip
\emph{Step 2: the induced involution on sections.}
Recall
\[
\OO_Q(1,1)=\pi_1^*\OO_{\PP^1}(1)\otimes\pi_2^*\OO_{\PP^1}(1).
\]
Define the antilinear involution on sections by
$
\widetilde\tau(s):=\tau^*(\overline s).
$
A direct computation on the basis gives
\begin{equation*}\label{eq:tau-on-basis}
\widetilde\tau(z_0w_0)=z_1w_1,\quad
\widetilde\tau(z_0w_1)=-z_1w_0,\quad
\widetilde\tau(z_1w_0)=-z_0w_1,\quad
\widetilde\tau(z_1w_1)=z_0w_0.
\end{equation*}
A section
\[
s=a\,z_0w_0+b\,z_0w_1+c\,z_1w_0+d\,z_1w_1
\]
satisfies $\widetilde\tau(s)=s$ if and only if
$
d=\overline a$ and $ c=-\overline b.
$
Hence the $\widetilde\tau$-invariant sections are precisely those of
the form
\begin{equation*}\label{eq:invariant-sections}
s=a\,z_0w_0+b\,z_0w_1
-\overline b\,z_1w_0+\overline a\,z_1w_1,
\qquad a,b\in\C.
\end{equation*}
Writing $a=a_1+ia_2$ and $b=b_1+ib_2$ with $a_j,b_j\in\R$, one gets a
real $4$-dimensional fixed subspace, hence the projectivised fixed
locus is $\RP^3$, with embedding \eqref{eq:RP3-embedding}. By
construction its points parametrise the $\tau$-invariant divisors in
$|\OO_Q(1,1)|$.

\smallskip
\emph{Step 3: a real pencil.}
The sections $s_1,s_2$ in \eqref{eq:s1s2-def} are
$\widetilde\tau$-invariant, so they span a real pencil. Define
\begin{equation*}\label{eq:h-def}
h:Q\dashrightarrow\PP^1,\qquad
h\bigl([z_0:z_1],[w_0:w_1]\bigr)
:=[s_1:-s_2]
=\bigl[z_0w_0+z_1w_1\,:\,-(z_0w_1-z_1w_0)\bigr].
\end{equation*}
Its base locus is $\{s_1=s_2=0\}$, namely
\begin{equation}\label{eq:base-locus}
\bigl([1:i],[1:i]\bigr),\qquad
\bigl([1:-i],[1:-i]\bigr),
\end{equation}
and $h\circ\tau=\overline h$ follows immediately from the
$\widetilde\tau$-invariance of $s_1$ and $s_2$.
\end{proof}

\Cref{prop:real-structure-Q} identifies, for the chosen
antiholomorphic involution $\tau$, a real form
\[
\RP^3\subset \PP H^0(Q,\OO_Q(1,1)),
\]
namely the projectivisation of the $\widetilde\tau$-fixed real vector
space. Its points parametrise the $\tau$-invariant divisors in the
linear system $|\OO_Q(1,1)|$.
The pencil $h$ is one explicit $\tau$-equivariant real pencil inside
this $\RP^3$. It is convenient, but neither canonical nor unique. Its
base locus consists of the two conjugate non-real points
\eqref{eq:base-locus}.

\section{Ward bundles and charge across the pushout}\label{sec:instantons-pushout}

Ward bundles on the two blown-up twistor spaces are glued across the double locus. The resulting bundle satisfies a componentwise obstruction criterion, and its second Chern cycle and polarised charge are additive. The Ward condition also implies analytic triviality near $Q$, so no additional local twisting occurs along the neck.

\subsection{Ward bundles and gluing}%\label{subsec:ward-setup}

The notation of Section~\ref{sec:CH-pushout} is retained: blown-up twistor spaces
$f_i:\widetilde Z_i\to Z_i$, exceptional quadrics
$Q_i\simeq\PP^1\times\PP^1$, isomorphism exchanging the two rulings
$\sigma:Q_1\xrightarrow{\sim}Q_2$, SNC pushout
\(Z_0:=\widetilde Z_1\cup_Q\widetilde Z_2\), \(Q\simeq Q_1\simeq Q_2\),
and semistable smoothing $\pi:\mathcal Z\to\Delta$ with smooth total
space and local model $t=uv$ along $Q$. Let $j_i:Q_i\hookrightarrow
\widetilde Z_i$, $\phi_i:\widetilde Z_i\hookrightarrow Z_0$, and
$\iota:Q\hookrightarrow Z_0$ be the canonical closed immersions.

Let $\mathcal E_i$ be \emph{Ward bundles} on $Z_i$: holomorphic
vector bundles of rank $r$ that restrict trivially to every real
twistor line \cite{Ward77,WardWells90}. Set
\begin{equation*}\label{eq:Fi-def}
F_i:=f_i^*\mathcal E_i\qquad\text{on }\widetilde Z_i.
\end{equation*}
Since $f_i|_{Q_i}$ factors through $\ell_i$ and $\mathcal
E_i|_{\ell_i}$ is trivial, we have
\begin{equation}\label{eq:Fi-trivial-on-Qi}
F_i|_{Q_i}\simeq \mathcal O_{Q_i}^{\oplus r}.
\end{equation}
Identify $Q:=Q_1$ and pull back via $\sigma$ to write
$F_2|_Q:=\sigma^*(F_2|_{Q_2})$.

\begin{definition}%\label{def:gluing-iso}
A \emph{$\sigma$-gluing isomorphism} for $(F_1,F_2)$ is an
isomorphism of vector bundles on $Q$
\begin{equation*}\label{eq:vartheta-def}
\vartheta:F_1|_Q\xrightarrow{\;\sim\;}F_2|_Q.
\end{equation*}
\end{definition}

Given a triple $(F_1,F_2,\vartheta)$, define a sheaf $\mathcal F$ on
$Z_0$ by the gluing exact sequence
\begin{equation}\label{eq:equalizer-F}
0\longrightarrow \mathcal F
\longrightarrow \phi_{1*}F_1\oplus\phi_{2*}F_2
\xrightarrow{\ \rho_\vartheta\ }
\iota_*(F_1|_Q)
\longrightarrow 0,
\end{equation}
where
\begin{equation*}\label{eq:rho-vartheta}
\rho_\vartheta(s_1,s_2)=s_1|_Q-\vartheta^{-1}(s_2|_Q).
\end{equation*}

\begin{lemma}
\label{lem:glued-bundle-correct}
The sheaf $\mathcal F$ defined by the exact sequence \eqref{eq:equalizer-F} is locally
free of rank $r$, with $\phi_i^*\mathcal F\simeq F_i$ for $i=1,2$
and $\mathcal F|_Q$ identified with $F_1|_Q$ and $F_2|_Q$ via
$\vartheta$.
\end{lemma}

 \begin{proof}
The claim is local on $Z_0$. Away from the double locus $Q$, the pushout
$Z_0=\widetilde Z_1\cup_Q \widetilde Z_2$ coincides with one of the two
smooth branches, so there is nothing to prove. It is therefore enough
to work near a point of $Q$.
Since $Z_0$ is the Ferrand pushout of
\[
j_1:Q\hookrightarrow \widetilde Z_1,
\qquad
j_2:Q\hookrightarrow \widetilde Z_2,
\]
its structure sheaf is locally
\[
\mathcal O_{Z_0}
\simeq
\mathcal O_{\widetilde Z_1}\times_{\mathcal O_Q}\mathcal O_{\widetilde Z_2}.
\]
Choose local trivialisations
\[
F_1\simeq \mathcal O_{\widetilde Z_1}^{\oplus r},
\qquad
F_2\simeq \mathcal O_{\widetilde Z_2}^{\oplus r}.
\]
In these coordinates, the gluing isomorphism
\[
\vartheta:F_1|_Q\xrightarrow{\sim}F_2|_Q
\]
is represented by a matrix $g\in GL_r(\mathcal O_Q)$.
Because the restriction map
$
\mathcal O_{\widetilde Z_2}\twoheadrightarrow \mathcal O_Q
$
is surjective, after shrinking we may lift $g$ to a matrix
$
\widetilde g\in M_r(\mathcal O_{\widetilde Z_2}).
$
Since $\det(g)$ is invertible on $Q$, after shrinking once more we may
assume that $\det(\widetilde g)$ is invertible, hence
$
\widetilde g\in GL_r(\mathcal O_{\widetilde Z_2}).
$
Replacing the chosen trivialisation of $F_2$ by $\widetilde g^{-1}$,
we reduce to the case $\vartheta=\mathrm{Id}$.
With this choice, \eqref{eq:equalizer-F} identifies $\mathcal F$
locally with
\[
\mathcal O_{\widetilde Z_1}^{\oplus r}
\times_{\mathcal O_Q^{\oplus r}}
\mathcal O_{\widetilde Z_2}^{\oplus r}
\cong
\left(
\mathcal O_{\widetilde Z_1}\times_{\mathcal O_Q}\mathcal O_{\widetilde Z_2}
\right)^{\oplus r}
\cong
\mathcal O_{Z_0}^{\oplus r}.
\]
Thus $\mathcal F$ is locally free of rank $r$.
The identifications $\phi_i^*\mathcal F\simeq F_i$ and the induced
identification of $\mathcal F|_Q$ with $F_1|_Q$ and $F_2|_Q$ via
$\vartheta$ follow from the construction.
\end{proof}

\begin{lemma}
\label{lem:ward-gluing-constant}
For Ward bundles \eqref{eq:Fi-trivial-on-Qi}, any gluing
isomorphism $\vartheta$ is constant: after choosing trivialisations of
$F_1|_Q$ and $F_2|_Q$, it is given by a constant matrix in
$GL_r(\C)$.
\end{lemma}

\begin{proof}
After trivialisations, $\vartheta$ is a morphism
$Q\simeq\PP^1\times\PP^1\to GL_r(\C)$. Every regular function on $Q$
is constant ($H^0(Q,\mathcal O_Q)=\C$), so the matrix entries are
constant.
\end{proof}

Up to changing the trivialisation on one side by a constant matrix, we may assume $\vartheta=\mathrm{Id}$ in the Ward case.

\subsection{Mayer--Vietoris and unobstructedness}
%\label{subsec:MV-unobs}

\begin{lemma}\label{lem:MV-End}
There is a short exact sequence of coherent sheaves on $Z_0$
\begin{equation}\label{eq:MV-End}
0\longrightarrow\mathrm{End}(\mathcal F)
\longrightarrow\phi_{1*}\mathrm{End}(F_1)\oplus\phi_{2*}\mathrm{End}(F_2)
\xrightarrow{\ \rho\ }
\iota_*\mathrm{End}(\mathcal F|_Q)
\longrightarrow 0,
\end{equation}
where $\rho(\varphi_1,\varphi_2)=\varphi_1|_Q-\mathrm{Ad}(\vartheta)(\varphi_2|_Q)$.
Taking cohomology yields the long exact sequence
\begin{equation}\label{eq:MV-End-LES}
\cdots\to H^k(Z_0,\mathrm{End}\,\mathcal F)\to
H^k(\widetilde Z_1,\mathrm{End}\,F_1)\oplus
H^k(\widetilde Z_2,\mathrm{End}\,F_2)\to
H^k(Q,\mathrm{End}(\mathcal F|_Q))\to\cdots.
\end{equation}
In the Ward case $\mathcal F|_Q\simeq\mathcal O_Q^{\oplus r}$, so
\begin{equation}\label{eq:H1H2Q-vanish}
H^1(Q,\mathrm{End}(\mathcal F|_Q))=0,\qquad
H^2(Q,\mathrm{End}(\mathcal F|_Q))=0.
\end{equation}
\end{lemma}

\begin{proof}
We first prove the exactness of \eqref{eq:MV-End}. Recall from
\eqref{eq:equalizer-F} that $\mathcal F$ is defined by gluing $F_1$
and $F_2$ along the isomorphism
\[
\vartheta:F_1|_Q\xrightarrow{\sim}F_2|_Q.
\]
An endomorphism of the glued bundle $\mathcal F$ is therefore
equivalent to a pair of endomorphisms
\[
\varphi_i\in \End(F_i),\qquad i=1,2,
\]
whose restrictions to the double locus are compatible with the
gluing, namely
\[
\varphi_1|_Q=\Ad(\vartheta)(\varphi_2|_Q)
=\vartheta\circ \varphi_2|_Q\circ \vartheta^{-1}.
\]
This identifies $\End(\mathcal F)$ with the kernel of the morphism
\[
\rho:\phi_{1*}\End(F_1)\oplus\phi_{2*}\End(F_2)
\longrightarrow \iota_*\End(\mathcal F|_Q),
\]
defined by
\[
\rho(\varphi_1,\varphi_2)
=
\varphi_1|_Q-\Ad(\vartheta)(\varphi_2|_Q).
\]
It remains to check surjectivity of $\rho$, which is a local question
near $Q$. Since $F_1$ is locally free on $\widetilde Z_1$, the
restriction map
\[
\End(F_1)\longrightarrow \End(F_1|_Q)
\]
is surjective as a morphism of sheaves along the closed immersion
$j_1:Q\hookrightarrow \widetilde Z_1$. Hence, given a local section
\[
\psi\in \End(\mathcal F|_Q)\simeq \End(F_1|_Q),
\]
we may choose a local lift $\widetilde\psi\in \End(F_1)$, and then
$
\rho(\widetilde\psi,0)=\psi.
$
Thus \eqref{eq:MV-End} is a short exact sequence.
The long exact sequence \eqref{eq:MV-End-LES} is then obtained by
applying sheaf cohomology to \eqref{eq:MV-End}.
In the Ward case, one has
\[
\mathcal F|_Q\simeq \mathcal O_Q^{\oplus r},
\]
hence
$
\End(\mathcal F|_Q)\simeq \End(\mathcal O_Q^{\oplus r})
\simeq \mathcal O_Q^{\oplus r^2}.
$
Therefore
$
H^k\bigl(Q,\End(\mathcal F|_Q)\bigr)
\simeq H^k(Q,\mathcal O_Q)^{\oplus r^2}.
$
Since $Q\simeq \PP^1\times \PP^1$, one has
\[
H^1(Q,\mathcal O_Q)=0,\qquad H^2(Q,\mathcal O_Q)=0,
\]
for instance by the K\"unneth formula and the vanishing
$H^1(\PP^1,\mathcal O_{\PP^1})=0$. This proves
\eqref{eq:H1H2Q-vanish}.
\end{proof}

\begin{proposition}\label{prop:DF-unobs-clean}
Assume
\begin{equation}\label{eq:H2-EndFi-zero}
H^2(\widetilde Z_1,\mathrm{End}\,F_1)=0,\qquad
H^2(\widetilde Z_2,\mathrm{End}\,F_2)=0.
\end{equation}
Then
$
H^2(Z_0,\mathrm{End}\,\mathcal F)=0.
$
\end{proposition}

\begin{proof}
Apply cohomology to the short exact sequence \eqref{eq:MV-End}. By
\Cref{lem:MV-End}, in the Ward case one has
\[
H^1\bigl(Q,\End(\mathcal F|_Q)\bigr)=0,
\qquad
H^2\bigl(Q,\End(\mathcal F|_Q)\bigr)=0.
\]
Hence the relevant portion of the long exact sequence \eqref{eq:MV-End-LES} becomes
\begin{multline*}
H^1\!\bigl(Q,\End(F|_Q)\bigr)\longrightarrow
H^2\!\bigl(Z_0,\End(F)\bigr)\longrightarrow
H^2\!\bigl(\widetilde Z_1,\End(F_1)\bigr)\\
{}\oplus H^2\!\bigl(\widetilde Z_2,\End(F_2)\bigr)
\longrightarrow H^2\!\bigl(Q,\End(F|_Q)\bigr).
\end{multline*}
Since the outer terms vanish, this simplifies to
\[
0\longrightarrow
H^2\!\bigl(Z_0,\End(F)\bigr)\longrightarrow
H^2\!\bigl(\widetilde Z_1,\End(F_1)\bigr)\oplus
H^2\!\bigl(\widetilde Z_2,\End(F_2)\bigr)
\longrightarrow 0.
\]
In particular, the middle arrow is injective, so there is a natural injection
\[
H^2\!\bigl(Z_0,\End(F)\bigr)\hookrightarrow
H^2\!\bigl(\widetilde Z_1,\End(F_1)\bigr)\oplus
H^2\!\bigl(\widetilde Z_2,\End(F_2)\bigr).
\]
By the assumption \eqref{eq:H2-EndFi-zero}, the right-hand side is zero. Therefore
\[
H^2\!\bigl(Z_0,\End(F)\bigr)=0,
\]
as claimed.
\end{proof}

The vanishing in \Cref{prop:DF-unobs-clean} concerns deformations of $\mathcal F$ on the fixed SNC threefold $Z_0$; it does not by itself imply the existence of a relative bundle on a semistable smoothing. In the Donaldson--Friedman setting, such a relative bundle is supplied by \cite[\S6.3]{DF}.

\subsection{Formal and analytic triviality near \texorpdfstring{$Q$}{Q}}
%\label{subsec:formal-analytic-triv}

The triviality in \eqref{eq:Fi-trivial-on-Qi} extends to every
infinitesimal neighbourhood of $Q_i$ and, by Grauert's formal principle
\cite{Grauert62}, to an analytic neighbourhood. Thus no local twisting
occurs near the double locus for Ward bundles.

For $i=1,2$ let $I_i\subset\mathcal O_{\widetilde Z_i}$ be the ideal
sheaf of $Q_i$ and set $Q_i^{(m)}:=V(I_i^{m+1})$.

\begin{lemma}\label{lem:conormal-powers}
One has
\begin{equation*}\label{eq:I-powers}
I_i^m/I_i^{m+1}\simeq\mathcal O_{Q_i}(m,m)
\qquad\text{for all }m\ge 0.
\end{equation*}
\end{lemma}

\begin{proof}
The conormal sheaf of $Q_i$ in $\widetilde Z_i$ is the dual of the
normal bundle:
\[
I_i/I_i^2\simeq\mathcal N_{Q_i/\widetilde Z_i}^\vee.
\]
By \Cref{lem:normal-exceptional},
$\mathcal N_{Q_i/\widetilde Z_i}\simeq\mathcal O_{Q_i}(-1)$, hence
$
I_i/I_i^2\simeq\mathcal O_{Q_i}(1)\simeq\mathcal O_{Q_i}(1,1).
$
Since $I_i$ is an invertible sheaf,
\[
I_i^m/I_i^{m+1}
\simeq(I_i/I_i^2)^{\otimes m}
\simeq\mathcal O_{Q_i}(m,m).
\qedhere
\]
\end{proof}

\begin{lemma}\label{lem:formal-triviality-neighborhoods}
Let $E$ be a holomorphic vector bundle on $\widetilde Z_i$ with
$E|_{Q_i}\simeq\mathcal O_{Q_i}^{\oplus r}$. Then for every $m\ge 0$,
\begin{equation*}\label{eq:formal-trivial}
E|_{Q_i^{(m)}}\simeq\mathcal O_{Q_i^{(m)}}^{\oplus r},
\end{equation*}
compatibly with the trivialisation on $Q_i$.
\end{lemma}

\begin{proof}
We argue by induction on $m$.
For $m=0$ there is nothing to prove, since the hypothesis is exactly
\[
E|_{Q_i}\simeq \mathcal O_{Q_i}^{\oplus r}
=
\mathcal O_{Q_i^{(0)}}^{\oplus r}.
\]
Now fix $m\ge 0$ and assume that a trivialisation
\[
E|_{Q_i^{(m)}}\simeq \mathcal O_{Q_i^{(m)}}^{\oplus r}
\]
has already been constructed. We want to extend it to
$Q_i^{(m+1)}$.
Consider the square-zero extension
\[
0\longrightarrow I_i^{m+1}/I_i^{m+2}
\longrightarrow \mathcal O_{Q_i^{(m+1)}}
\longrightarrow \mathcal O_{Q_i^{(m)}}
\longrightarrow 0.
\]
The obstruction to lifting the trivial bundle
$\mathcal O_{Q_i^{(m)}}^{\oplus r}$ from $Q_i^{(m)}$ to
$Q_i^{(m+1)}$ lies in
\[
H^1\!\left(
Q_i,\,
\End(\mathcal O_{Q_i}^{\oplus r})\otimes
I_i^{m+1}/I_i^{m+2}
\right).
\]
Since
\[
\End(\mathcal O_{Q_i}^{\oplus r})
\simeq \mathcal O_{Q_i}^{\oplus r^2},
\]
and by \Cref{lem:conormal-powers} one has
\[
I_i^{m+1}/I_i^{m+2}\simeq \mathcal O_{Q_i}(m+1,m+1),
\]
this obstruction group is
\[
H^1\!\left(Q_i,\mathcal O_{Q_i}(m+1,m+1)\right)^{\oplus r^2}.
\]

It remains to prove that
$
H^1\!\left(Q_i,\mathcal O_{Q_i}(m+1,m+1)\right)=0.
$
Since $Q_i\simeq \PP^1\times \PP^1$, the K\"unneth formula
\cite[III.5.7]{Hartshorne} gives
\[
H^1\!\left(\PP^1\times\PP^1,\mathcal O(m+1,m+1)\right)
\cong
\bigl(
H^1(\PP^1,\mathcal O(m+1))\otimes
H^0(\PP^1,\mathcal O(m+1))
\bigr)
\]
\[
\qquad\qquad\oplus
\bigl(
H^0(\PP^1,\mathcal O(m+1))\otimes
H^1(\PP^1,\mathcal O(m+1))
\bigr).
\]
But
$
H^1(\PP^1,\mathcal O(m+1))=0 $ for all $
  m\ge 0,$
so both summands vanish. Hence the obstruction group is zero, and the
given trivialisation extends from $Q_i^{(m)}$ to $Q_i^{(m+1)}$.
By induction, one obtains for every $m\ge 0$ an isomorphism
\[
E|_{Q_i^{(m)}}\simeq \mathcal O_{Q_i^{(m)}}^{\oplus r}.
\]
Moreover, the trivialisations can be chosen compatibly with the one on
$Q_i$, because at each step we are extending a previously fixed
trivialisation.
\end{proof}

\begin{lemma}\label{lem:analytic-triviality-near-Q}
If $E$ is a holomorphic vector bundle on $\widetilde Z_i$ with
$E|_{Q_i}\simeq\mathcal O_{Q_i}^{\oplus r}$, then there exists an
analytic neighbourhood $U_i$ of $Q_i$ in $\widetilde Z_i$ such that
$E|_{U_i}\simeq\mathcal O_{U_i}^{\oplus r}$.
\end{lemma}

\begin{proof}
By \Cref{lem:formal-triviality-neighborhoods}, for every $m\ge 0$
one has
\[
E|_{Q_i^{(m)}}\simeq \mathcal O_{Q_i^{(m)}}^{\oplus r},
\]
compatibly in $m$. Equivalently, the formal completion of $E$ along
$Q_i$ is trivial:
\[
\widehat E
:=
E|_{\widehat{(\widetilde Z_i)}_{Q_i}}
\simeq
\mathcal O_{\widehat{(\widetilde Z_i)}_{Q_i}}^{\oplus r}.
\]

Now $Q_i\subset \widetilde Z_i$ is the exceptional divisor of the
blow-up of a smooth threefold along a smooth curve, hence it is an
exceptional analytic set in the sense of Grauert. Moreover, by \Cref{lem:normal-exceptional}, its normal bundle is
$\mathcal N_{Q_i/\widetilde Z_i}\simeq \mathcal O_{Q_i}(-1,-1)$ and is negative. Grauert's comparison/formal principle therefore applies to coherent locally free
sheaves in a sufficiently small neighbourhood of $Q_i$
\cite{Grauert62}. In particular, an isomorphism of formal completions
along $Q_i$ comes from an analytic isomorphism on some neighbourhood
of $Q_i$. Applying this to the two locally free sheaves
$E$ and $\mathcal O_{\widetilde Z_i}^{\oplus r}$, whose formal
completions are isomorphic, we conclude that there exists a sufficiently
small analytic neighbourhood $U_i$ of $Q_i$ in $\widetilde Z_i$ such
that
\[
E|_{U_i}\simeq \mathcal O_{U_i}^{\oplus r}.
\]

Thus $E$ is holomorphically trivial on an analytic neighbourhood of
$Q_i$.
\end{proof}

Thus a Ward bundle is holomorphically trivial on a neighbourhood of $Q_i$. This analytic triviality explains the absence of local twisting in the charge computation of \Cref{prop:no-neck-term-ward}.

\subsection{Additivity of the polarised charge}
%\label{subsec:charge-additivity}

Fix a Cartier class $H\in\CH^1(\mathcal Z)$ such that
$H_t:=H|_{Z_t}$ is ample for $t\neq 0$, and write
$H_0:=H|_{Z_0}$.

\begin{definition}\label{def:polarised-charge-clean}
For $t\neq 0$ define
\begin{equation*}\label{eq:charge-def}
\mathrm{charge}_H(\mathcal F_t)
:=\deg\bigl((c_2(\mathcal F_t)\,H_t)\cap[Z_t]\bigr)\in\Z.
\end{equation*}
\end{definition}

\begin{proposition}\label{prop:charge-specialises}
Assume there exists a holomorphic vector bundle $\mathscr F$ on
$\mathcal Z$, flat over $\Delta$, such that
\[
\mathscr F|_{Z_0}\simeq \mathcal F
\qquad\text{and}\qquad
\mathscr F|_{Z_t}\simeq \mathcal F_t
\]
for all $t\neq 0$ sufficiently small. Then
\begin{equation}\label{eq:charge-constant}
\mathrm{charge}_H(\mathcal F_t)
= \deg\bigl((c_2(\mathcal F)\,H_0)\cap[Z_0]\bigr)
\qquad\text{for }t\neq 0\text{ small}.
\end{equation}
\end{proposition}

\begin{proof}
Base-change the family to the discrete valuation ring $R=\C[[t]]$
as in Subsection~\ref{subsec:chow-specialisation}. Since $\mathscr F$
is locally free on the smooth total space $\mathcal Z$, ordinary and
operational Chern classes coincide on $\mathcal Z$
\cite[Ch.~17]{Fulton}; in particular $c_2(\mathscr F)\in A^2(\mathcal Z)$
is a well-defined operational class. The Cartier class
$H\in\CH^1(\mathcal Z)$ may be treated in the same bivariant
formalism.
Fulton's specialisation map \cite[\S20.3]{Fulton} sends the
fundamental class of the generic fibre to the fundamental class of
the special fibre. By functoriality of operational Chow
\cite[Ch.~17]{Fulton}, specialisation commutes with the operator
\[
\alpha\longmapsto\bigl(c_2(\mathscr F)\,H\bigr)\cap\alpha.
\]
Applying this to the fundamental class of the generic fibre gives
\[
\operatorname{sp}\!\left((c_2(\mathcal F_t)\,H_t)\cap[Z_t]\right)
=
(c_2(\mathcal F)\,H_0)\cap[Z_0].
\]
Taking degrees yields \eqref{eq:charge-constant}.
\end{proof}

\begin{remark}\label{rem:ward-hypotheses}
Suppose a relative bundle $\mathscr F$ extends $\mathcal F$ and, for
$t\neq0$ small, $\mathcal F_t:=\mathscr F|_{Z_t}$ is trivial on every
real twistor line, carries the compatible quaternionic structure, and
has trivial determinant. Then the Ward--Atiyah correspondence
\cite{Ward77,WardWells90,AtiyahWard1977} gives an anti-self-dual
$SU(r)$-connection on $X_t\simeq X_1\#X_2$. The existence of the relative bundle is not a formal consequence of the fixed-fibre vanishing in \Cref{prop:DF-unobs-clean}; in the Donaldson--Friedman setting it follows from the deformation argument of \cite[\S6.3]{DF}.
\end{remark}

Consider 
\begin{equation}\label{eq:gamma-def}
\gamma(\mathcal F):=c_2(\mathcal F)\cap[Z_0]\in\CH_1(Z_0),\qquad
\gamma_i:=c_2(F_i)\cap[\widetilde Z_i]\in\CH_1(\widetilde Z_i).
\end{equation}

\begin{theorem}
\label{prop:no-neck-term-ward}
Let $\mathcal F$ be the vector bundle on
$Z_0=\widetilde Z_1\cup_Q \widetilde Z_2$ obtained by gluing 
$(F_1,F_2,\vartheta)$ as in
\Cref{lem:glued-bundle-correct}. With the notation
\eqref{eq:gamma-def}, one has
\begin{equation}\label{eq:gamma-additive}
\gamma(\mathcal F)=(\phi_1)_*\gamma_1+(\phi_2)_*\gamma_2
 \qquad\text{in }\CH_1(Z_0).
\end{equation}
Consequently, for any $H_0\in\CH^1(Z_0)$ and $H_i:=\phi_i^*H_0$,
\begin{equation}\label{eq:additivity-central}
 \deg\bigl((c_2(\mathcal F)\,H_0)\cap[Z_0]\bigr)
=
\deg\bigl((c_2(F_1)\,H_1)\cap[\widetilde Z_1]\bigr)
+
\deg\bigl((c_2(F_2)\,H_2)\cap[\widetilde Z_2]\bigr).
\end{equation}
If, moreover, $\mathcal F$ is extended by a relative bundle
$\mathscr F$ on $\mathcal Z$ as in
\Cref{prop:charge-specialises}, then for $t\neq 0$
sufficiently small,
\begin{equation}\label{eq:additivity-smooth}
\mathrm{charge}_H(\mathcal F_t)
=
\deg\bigl((c_2(F_1)\,H_1)\cap[\widetilde Z_1]\bigr)
+
\deg\bigl((c_2(F_2)\,H_2)\cap[\widetilde Z_2]\bigr).
\end{equation}
\end{theorem}

\begin{proof}
Since $Z_0$ is the reduced union of the two irreducible components
$\widetilde Z_1$ and $\widetilde Z_2$, one has
\[
[Z_0]=(\phi_1)_*[\widetilde Z_1]+(\phi_2)_*[\widetilde Z_2]
\qquad\text{in }\CH_3(Z_0).
\]
Therefore
\[
\gamma(\mathcal F)
=c_2(\mathcal F)\cap[Z_0]
=c_2(\mathcal F)\cap(\phi_1)_*[\widetilde Z_1]
+c_2(\mathcal F)\cap(\phi_2)_*[\widetilde Z_2].
\]
By the projection formula for operational Chow groups
\cite[Ch.~17]{Fulton},
\[
c_2(\mathcal F)\cap(\phi_i)_*[\widetilde Z_i]
=(\phi_i)_*\bigl(\phi_i^*c_2(\mathcal F)\cap[\widetilde Z_i]\bigr).
\]
Since $\phi_i^*\mathcal F\simeq F_i$, functoriality of Chern classes
gives $\phi_i^*c_2(\mathcal F)=c_2(F_i)$, hence
\[
\gamma(\mathcal F)
=(\phi_1)_*\gamma_1+(\phi_2)_*\gamma_2,
\]
which is \eqref{eq:gamma-additive}. Now, 
capping \eqref{eq:gamma-additive} with $H_0$ and applying the
projection formula gives \eqref{eq:additivity-central}, and
\eqref{eq:additivity-smooth} follows from
\Cref{prop:charge-specialises}.
\end{proof}

For Ward bundles, the analytic triviality near $Q$ shows that no
additional local twisting contributes to \eqref{eq:gamma-additive}.

\section{Instantons from Hartshorne--Serre curves on the pushout}
\label{sec:serre-pushout-instantons}

The second bundle construction concerns rank-$2$ bundles obtained by the Hartshorne--Serre construction on the two branches. Under \Cref{ass:HS-clean}, their restrictions to $Q$ are trivial; hence they glue to a bundle on $Z_0$, and both the second Chern cycle and the polarised charge are additive.

\subsection{Hartshorne--Serre bundles on the pushout}
%\label{subsec:HS-data}

Let $C_i\subset\widetilde Z_i$ be pure codimension-$2$ local complete intersection curves.
Fix line bundles $L_i$ on $\widetilde Z_i$ and assume $(C_i,L_i)$
satisfy the Hartshorne--Serre hypotheses, so there exists a rank-$2$
vector bundle $E_i$ fitting into
\begin{equation}\label{eq:HS}
0\longrightarrow\mathcal O_{\widetilde Z_i}\longrightarrow E_i
\longrightarrow I_{C_i}\otimes L_i\longrightarrow 0.
\end{equation}
The Chern classes are
\begin{equation*}\label{eq:HS-Chern}
c_1(E_i)=c_1(L_i),\qquad c_2(E_i)=[C_i]\in\CH^2(\widetilde Z_i).
\end{equation*}
The gluing result is proved under the following hypothesis.

\begin{assumption} \label{ass:HS-clean}
Assume $L_i\simeq\mathcal O_{\widetilde Z_i}$ and
$C_i\cap Q=\varnothing$ for $i=1,2$.
\end{assumption}

Under \Cref{ass:HS-clean}, restricting
\eqref{eq:HS} to $Q$, and using that $L_i|_Q\simeq\mathcal O_Q$ and
$C_i\cap Q=\varnothing$, one obtains an exact sequence
\[
0\longrightarrow \mathcal O_Q
\longrightarrow E_i|_Q
\longrightarrow \mathcal O_Q
\longrightarrow 0.
\]
Its extension class lies in
$
\Ext^1_{\mathcal O_Q}(\mathcal O_Q,\mathcal O_Q)
\simeq H^1(Q,\mathcal O_Q)=0,
$
hence the sequence splits and therefore
$
E_i|_Q\simeq \mathcal O_Q^{\oplus 2}.
$
Fix an isomorphism
$$
\varphi:E_1|_Q\xrightarrow{\sim}E_2|_Q.
$$

\begin{proposition} 
\label{prop:glued-E0}
There exists a rank-$2$ vector bundle $E_0$ on $Z_0$ such that
$\phi_i^*E_0\simeq E_i$ for $i=1,2$, with the identification on $Q$
given by $\varphi$.
\end{proposition}

\begin{proof}
Since
\(Z_0=\widetilde Z_1\cup_Q \widetilde Z_2\)
is the Ferrand pushout of the closed immersions
\(j_i:Q\hookrightarrow \widetilde Z_i\),
a vector bundle on $Z_0$ is obtained by gluing vector bundles on the
two branches together with an identification of their restrictions to
the double locus.
Starting from the isomorphism
\[
\varphi:E_1|_Q\xrightarrow{\sim}E_2|_Q,
\]
consider the morphism
\[
\rho_\varphi:\phi_{1*}E_1\oplus\phi_{2*}E_2
\longrightarrow \iota_*(E_1|_Q)
\]
defined by
$
\rho_\varphi(s_1,s_2)
=
s_1|_Q-\varphi^{-1}(s_2|_Q),$
where $\phi_i:\widetilde Z_i\hookrightarrow Z_0$ and
$\iota:Q\hookrightarrow Z_0$ are the natural closed immersions.
Define
$
E_0:=\ker(\rho_\varphi).
$
Then $E_0$ fits into the exact sequence
\begin{equation*}
0\longrightarrow E_0
\longrightarrow \phi_{1*}E_1\oplus\phi_{2*}E_2
\xrightarrow{\ \rho_\varphi\ }
\iota_*(E_1|_Q)
\longrightarrow 0.
\end{equation*}
It remains to show that $E_0$ is locally free of rank $2$.
This is a local statement near the double locus $Q$. Away from $Q$,
the space $Z_0$ coincides with one of the two smooth branches, so
$E_0$ is evidently identified with $E_1$ or $E_2$ there.
Near a point of $Q$, the pushout structure gives
\[
\mathcal O_{Z_0}
\simeq
\mathcal O_{\widetilde Z_1}\times_{\mathcal O_Q}\mathcal O_{\widetilde Z_2}.
\]
After choosing local trivialisations
\[
E_1\simeq \mathcal O_{\widetilde Z_1}^{\oplus 2},
\qquad
E_2\simeq \mathcal O_{\widetilde Z_2}^{\oplus 2},
\]
the isomorphism $\varphi$ is represented by a matrix
$
g\in GL_2(\mathcal O_Q).
$
Since the restriction map
$
\mathcal O_{\widetilde Z_2}\twoheadrightarrow \mathcal O_Q
$
is surjective, after shrinking the neighbourhood we may lift the
entries of $g$ to a matrix $\widetilde g\in M_2(\mathcal O_{\widetilde Z_2})$.
Because $\det(g)\in \mathcal O_Q^\times$, after shrinking once more
we may assume that $\det(\widetilde g)$ is invertible, so that
$
\widetilde g\in GL_2(\mathcal O_{\widetilde Z_2}).
$
Changing the chosen trivialisation of $E_2$ by $\widetilde g^{-1}$,
we may reduce to the case $\varphi=\mathrm{Id}$.
In this gauge, the glued module becomes
\[
\mathcal O_{\widetilde Z_1}^{\oplus 2}
\times_{\mathcal O_Q^{\oplus 2}}
\mathcal O_{\widetilde Z_2}^{\oplus 2}
\cong
\left(
\mathcal O_{\widetilde Z_1}\times_{\mathcal O_Q}\mathcal O_{\widetilde Z_2}
\right)^{\oplus 2}
\cong
\mathcal O_{Z_0}^{\oplus 2}.
\]
Hence $E_0$ is locally free of rank $2$ on $Z_0$.
By construction, its pull-back to each branch recovers the original
bundle:
\[
\phi_i^*E_0\simeq E_i,
\qquad i=1,2,
\]
and the induced identification on $Q$ is precisely $\varphi$.
\end{proof}

The arguments leading to \Cref{prop:c2-pushout-general,prop:HS-charge-additivity}
require only the associated one-cycle, not the full dependence of the
operational class $c_2(E_0)\in A^2(Z_0)$ on the gluing isomorphism
$\varphi$. No additivity of $c_2(E_0)$ as an element of $A^2(Z_0)$ is
asserted: the gluing automorphism $\varphi$ may in principle contribute
a correction term supported on the double locus $Q$. The projection
formula nevertheless gives an additive formula for
\[
\gamma(E_0):=c_2(E_0)\cap[Z_0]\in\CH_1(Z_0),
\]
which is independent of the choice of $\varphi$.

\subsection{Additivity of the second Chern cycle}
%\label{subsec:neck-correction}

Since $Z_0$ is singular, $c_2(E_0)$ is an operational class,
\[
c_2(E_0)\in A^2(Z_0).
\]
Set
\[
\gamma(E_0):=c_2(E_0)\cap[Z_0]\in\CH_1(Z_0).
\]

\begin{proposition}
\label{prop:c2-pushout-general}
Let $E_0$ be the rank-$2$ vector bundle on $Z_0$ glued from
$(E_1,E_2,\varphi)$ as in \Cref{prop:glued-E0}. Then
\begin{equation}\label{eq:gamma-E0-additive}
\gamma(E_0)
=
(\phi_1)_*\bigl(c_2(E_1)\cap[\widetilde Z_1]\bigr)
+
(\phi_2)_*\bigl(c_2(E_2)\cap[\widetilde Z_2]\bigr)
\qquad\text{in }\CH_1(Z_0).
\end{equation}
Under \Cref{ass:HS-clean}, this becomes
\begin{equation}\label{eq:gamma-E0-HS}
\gamma(E_0)=\phi_{1*}[C_1]+\phi_{2*}[C_2].
\end{equation}
\end{proposition}

\begin{proof}
Since $Z_0$ is the reduced union of the two irreducible components
$\widetilde Z_1$ and $\widetilde Z_2$, one has
\[
[Z_0]=(\phi_1)_*[\widetilde Z_1]+(\phi_2)_*[\widetilde Z_2]
\qquad\text{in }\CH_3(Z_0).
\]
Therefore
\[
\gamma(E_0)
=
c_2(E_0)\cap[Z_0]
=
c_2(E_0)\cap(\phi_1)_*[\widetilde Z_1]
+
c_2(E_0)\cap(\phi_2)_*[\widetilde Z_2].
\]
By the projection formula for operational Chow groups,
\[
c_2(E_0)\cap(\phi_i)_*[\widetilde Z_i]
=
(\phi_i)_*\bigl(\phi_i^*c_2(E_0)\cap[\widetilde Z_i]\bigr).
\]
Since $\phi_i^*E_0\simeq E_i$, functoriality of Chern classes gives
$
\phi_i^*c_2(E_0)=c_2(E_i).
$
Hence
\[
\gamma(E_0)
=
(\phi_1)_*\bigl(c_2(E_1)\cap[\widetilde Z_1]\bigr)
+
(\phi_2)_*\bigl(c_2(E_2)\cap[\widetilde Z_2]\bigr),
\]
which is \eqref{eq:gamma-E0-additive}. Under
\Cref{ass:HS-clean}, one has $c_2(E_i)=[C_i]$ on the smooth
threefold $\widetilde Z_i$, so \eqref{eq:gamma-E0-HS} follows.
\end{proof}

For specialisation and charge, only the cycle $\gamma(E_0)=c_2(E_0)\cap[Z_0]\in\CH_1(Z_0)$ is used; no separate description of the operational class $c_2(E_0)\in A^2(Z_0)$ is required.

\subsection{Deformation, charge, and Ward instantons}
%\label{subsec:deformation-charge}

We apply \Cref{prop:charge-specialises} to $E_0$.

\begin{assumption}\label{ass:DEF}
There exists a vector bundle $E$ on $\mathcal Z$, flat over $\Delta$,
with $E|_{Z_0}\simeq E_0$ and $E|_{Z_t}\simeq E_t$ for $t\ne 0$.
\end{assumption}

\begin{theorem}
\label{prop:HS-charge-additivity}
Under \Cref{ass:HS-clean} and \Cref{ass:DEF}, for $t\ne 0$
sufficiently small,
\begin{equation}\label{eq:charge-specialises}
\mathrm{charge}_H(E_t)
=\deg\bigl((c_2(E_0)\,H_0)\cap[Z_0]\bigr).
\end{equation}
Moreover,
\begin{equation}\label{eq:charge-additive}
\mathrm{charge}_H(E_t)
=
\deg\bigl((c_2(E_1)\,H_1)\cap[\widetilde Z_1]\bigr)
+
\deg\bigl((c_2(E_2)\,H_2)\cap[\widetilde Z_2]\bigr),
\end{equation}
where $H_i:=\phi_i^*H_0$.
\end{theorem}

\begin{proof}
Equation \eqref{eq:charge-specialises} follows from \Cref{prop:charge-specialises} applied to $E_0$.
By \Cref{prop:c2-pushout-general},
\[
c_2(E_0)\cap[Z_0]
=
(\phi_1)_*\bigl(c_2(E_1)\cap[\widetilde Z_1]\bigr)
+
(\phi_2)_*\bigl(c_2(E_2)\cap[\widetilde Z_2]\bigr).
\]
Capping with $H_0$ and using the projection formula gives
\[
\deg\bigl((c_2(E_0)\,H_0)\cap[Z_0]\bigr)
=
\deg\bigl((c_2(E_1)\,H_1)\cap[\widetilde Z_1]\bigr)
+
\deg\bigl((c_2(E_2)\,H_2)\cap[\widetilde Z_2]\bigr),
\]
which is \eqref{eq:charge-additive}.
\end{proof}

\begin{corollary} 
\label{cor:instantons-HS}
Assume that, for all sufficiently small \(t\neq 0\), the bundle
\(E_t\) on \(Z_t\) satisfies the following conditions:
\begin{enumerate}
\item \(E_t\) restricts trivially to every real twistor line in \(Z_t\);
\item \(E_t\) carries a quaternionic structure compatible with the real
structure of \(Z_t\);
\item \(\det E_t\simeq \OO_{Z_t}\).
\end{enumerate}
Then, by the Ward--Atiyah correspondence
\cite{Ward77,WardWells90,AtiyahWard1977}, the bundle \(E_t\)
determines an anti-self-dual \(SU(2)\)-connection on the underlying
oriented Riemannian \(4\)-manifold
\[
X_t\simeq X_1\# X_2.
\]
Moreover, its polarised charge is given by
\eqref{eq:charge-additive}.
\end{corollary}

\begin{proof}
The three assumptions are the standard hypotheses under which the Ward transform associates to \(E_t\) an anti-self-dual
\(SU(2)\)-connection on \(X_t\); see
\cite{Ward77,WardWells90,AtiyahWard1977}. The formula for the
polarised charge is given by \eqref{eq:charge-additive}.
\end{proof}

In the Donaldson--Friedman construction, the bundle-deformation argument of \cite[\S6.3]{DF} ensures persistence of the Ward--Atiyah hypotheses along the smoothing.

\paragraph{Further questions.}
The Hartshorne--Serre construction in \Cref{ass:HS-clean} assumes that the curves do not meet $Q$. For curves meeting the double locus, the restricted
fixed-phase bundle of \Cref{prop:KN-circle-on-c} provides additional
boundary information not detected by the Chow class alone. It remains to determine how this phase information enters a more
general gauge-theoretic gluing problem.

\section*{Acknowledgements}
Both authors are partially supported by the PRIN 2022MWPMAB ``Interactions between Geometric
Structures and Function Theories'', and by INdAM-GNSAGA.

\end{document}